\documentclass[a4paper,12pt,reqno]{amsart}
\usepackage{amsmath, amssymb, amsthm, color}
\numberwithin{equation}{section}
\newtheorem{theorem}{Theorem}[section]
\newtheorem{proposition}[theorem]{Proposition}
\newtheorem{corollary}[theorem]{Corollary}
\newtheorem{lemma}[theorem]{Lemma}
\theoremstyle{definition}
\newtheorem*{remark}{Remark}
\newcommand{\vep}{\varepsilon}
\newcommand{\R}{{\bf R}}
\newcommand{\N}{{\bf N}}
\newcommand{\dsp}{\displaystyle}
\newcommand{\lam}{\lambda}
\newcommand{\cal}{\mathcal}
\begin{document}

\title{
Existence of infinitely many singular self-similar solutions 
for the semilinear heat equation}

\author{Y\={u}ki Naito}
\address[Y\={u}ki Naito]{
Department of Mathematics,  Hiroshima University \\
Higashi-Hiroshima, 739-8526, Japan}
\email{yunaito@hiroshima-u.ac.jp}

\date{\today}

\begin{abstract}
We show the existence of infinitely many singular positive forward self-similar solutions for the 
semilinear heat equation $\partial_t u = \Delta u + u^p$ for $x \in \R^N$, $t > 0$,  
where $N \geq 3$ and $N/(N-2) \leq p \leq (N+2)/(N-2)$. 
We also investigate the asymptotic behavior and the bound of the profiles of the solutions as $|x| \to 0$ 
and $|x| \to \infty$, respectively.  
Furthermore, we show nonuniqueness of positive solutions of the Cauchy problem for 
the equation with $p = N/(N-2)$ by using singular self-similar solutions. 
\end{abstract}



\maketitle

\setlength{\baselineskip}{5.5mm}

\section{Introduction}

In this paper we study the singular positive self-similar solutions of the equation 
\begin{equation}
	\partial_t w = \Delta w + w^p, \quad x \in \R^N, \ t > 0,
	\label{eq1.1}
\end{equation}
where $N \geq 3$ and $p \geq N/(N-2)$. 
The equation (\ref{eq1.1}) is invariant under the similarity transformation
$$
	w(x,t) \mapsto w_{\lam}(x, t) = \lam^{2/(p-1)}w(\lam x, \lam^2 t) 
	\quad \mbox{for all} \ \lam > 0.
$$
In particular, a solution $w$ is said to be self-similar when $w = w_{\lam}$ for all $\lam > 0$. 
Such self-similar solutions are global in time and 
expected to describe the large time behavior of global solutions to (\ref{eq1.1}). 
It can be easily checked that $w$ is a forward self-similar solution 
if and only if $w$ has the form
$$
	w(x, t) = t^{-1/(p-1)}u(x/\sqrt{t}) \quad \mbox{for} \ x \in \R^N, \ t > 0,
$$
where the profile function $u$ satisfies the elliptic equation
$$
	\Delta u + \frac{1}{2}x\cdot \nabla u + \frac{1}{p-1}u + u^p = 0 
	\quad \mbox{in} \ \R^N.
$$
In particular, if $u = u(r)$, $r = |x|$, is radially symmetric about the origin, 
then $u$ satisfies the ordinary differential equation
\begin{equation}
	u'' + \left(\frac{N-1}{r} + \frac{r}{2}\right)u' + \frac{1}{p-1}u + u^p = 0 
	\quad \mbox{for} \ r > 0.
	\label{eq1.2}
\end{equation}
Equation (\ref{eq1.2}) was introduced in \cite{HaWe}, and 
the properties of regular solutions of (\ref{eq1.2}) have been studied extensively 
in \cite{EsKa, Naib, PTW, Weia, Weib}. 
These solutions are often used to study the large time behavior of
global solutions to the equation \cite{CaWe, Kav, Kaw, Naic, Naid, STW}  
and to study solutions of the Cauchy problem of (\ref{eq1.1}) with 
singular initial data \cite{Naia, SoWe}.

In this paper we consider the existence and properties of singular positive solutions of (\ref{eq1.2}).  
By a singular positive solution $u$ of (\ref{eq1.2}), 
we mean that $u \in C^2(0, \infty)$ is a positive solution of (\ref{eq1.2}) for $r > 0$ 
and it satisfies $u(r) \to \infty$ as $r \to 0$. 

It is well known that, in the case $p > N/(N-2)$, 
equation (\ref{eq1.1}) has a singular positive stationary solution 
\begin{equation}
	U_L(|x|) = L|x|^{-2/(p-1)}, \quad \mbox{where} \  
	L = \left\{\frac{2}{p-1}\left(N-2-\frac{2}{p-1}\right)\right\}^{1/(p-1)}.
	\label{eq1.3}
\end{equation}
Let us mention that $U_L(r)$ is also a singular positive solution of (\ref{eq1.2}). 
It was shown by Quittner \cite{Qui} that, if $u$ 
is a singular positive self-similar solution of (\ref{eq1.2})
for $r > 0$, then $u(r) \equiv U_L(r)$ provided $p > (N+2)/(N-2)$. 
On the other hand, in the case 
$$
	\frac{N}{N-2} < p < \left\{
	\begin{array}{cl}
	\frac{N+ 2\sqrt{N-1}}{N-4+2\sqrt{N-1}}, \quad & \mbox{if} \ N \leq 10,
	\\[2ex]
	\frac{N+2}{N-1}, \quad  & \mbox{if} \ N > 10,
	\end{array}
	\right.
$$
the existence of a continuum of singular positive solutions was shown by Sato \cite{Sat}.
We refer to \cite[Appendix Ga]{QuSo}, 
in which a summary on the existence and nonexistence of singular self-similar solutions is provided.
In this paper, we consider the case $N/(N-2) \leq p \leq (N+2)/(N-2)$. 
Our first result is the following.

\begin{theorem}
\label{thm1.1}
Let $N/(N-2) \leq p < (N+2)/(N-2)$. 
Then $(\ref{eq1.2})$ has a continuum of singular positive solutions. 
In the case $N/(N-2) < p < (N+2)/(N-2)$, 
any singular positive solution $u$ satisfies 
\begin{equation}
	\lim_{r \to 0}r^{2/(p-1)}u(r) = L, 
	\label{eq1.4}
\end{equation}
where $L$ is the constant defined in $(\ref{eq1.3})$. 
In the case $p = N/(N-2)$, any singular positive solution $u$ satisfies 
\begin{equation}
	\lim_{r \to 0}r^{N-2}(-(\log r))^{(N-2)/2}u(r) = \left(\frac{(N-2)^2}{2}\right)^{(N-2)/2}.
	\label{eq1.5}
\end{equation}
\end{theorem}

\begin{remark}
(i) In \cite{Sat} the singular solution was constructed as a convergence limit of 
approximate solutions in an appropriate function space. 
In the proof of Theorem \ref{thm1.1}, we use a shooting method to show the existence of singular solutions. 

(ii) Let $u^*$ be a singular positive solution of (\ref{eq1.2}) obtained by Theorem \ref{thm1.1}. 
Define $w^*(x, t)$ by 
\begin{equation}
	w^*(x, t) = t^{-1/(p-1)}u^*(|x|/\sqrt{t}) \quad \mbox{for} \ x \in \R^N\setminus\{0\}, \ t > 0.
	\label{eq1.6}
\end{equation}
Then $w^*$ is a singular positive self-similar solution of (\ref{eq1.1}). 
In the case $N/(N-2) < p < (N+2)/(N-2)$, we have 
$$
	|x|^{2/(p-1)}w^*(x, t) = (|x|/\sqrt{t})^{2/(p-1)}u^*(|x|/\sqrt{t}) \to L 
	\quad \mbox{as} \ |x| \to 0. 
$$
for each fixed $t > 0$.
Note that $\log(|x|/\sqrt{t}) \sim \log |x|$ as $|x| \to 0$ for each fixed $t > 0$.  
Then, in the case $p = N/(N-2)$, we have 
\begin{equation}
	|x|^{N-2}(-(\log |x|))^{(N-2)/2}w^*(x, t) 
	\to \left(\frac{(N-2)^2}{2}\right)^{(N-2)/2} 
	 \quad \mbox{as} \ |x| \to 0
	\label{eq1.7}
\end{equation}
for each fixed $t > 0$. 
\end{remark}

To state the results in the case $p = (N+2)/(N-2)$, 
we define 
\begin{equation}
	\Phi(v) = -\frac{L^{p-1}}{2}v^2 + \frac{1}{p+1}v^{p+1}
	\quad \mbox{for} \ v \geq 0,
	\label{eq1.8}
\end{equation}
where $L$ is the constant defined in (\ref{eq1.3}). 
Then $\Phi$ satisfies 
$\Phi(0) = \Phi(A_0) = 0$ and $\Phi(v) < 0$ for $0 < v < A_0$, where 
\begin{equation}
	A_0 = \left(\frac{(p+1)L^{p-1}}{2}\right)^{1/(p-1)}.
	\label{eq1.9}
\end{equation}
We see that 
\begin{equation}
	\Phi(v) \geq \Phi(L) = -\left(\frac{1}{2}-\frac{1}{p+1}\right)L^{p+1}
	\quad \mbox{for} \ v \geq 0.
	\label{eq1.10}
\end{equation}
We obtain the following.

\begin{theorem}
\label{thm1.2}
Let $p = (N+2)/(N-2)$. 
Then $(\ref{eq1.2})$ has infinitely many singular positive solutions. 
For each singular positive solution $u$, 
there exist constants $\gamma_1 \leq \gamma_2$ such that 
\begin{equation}
	\gamma_1 = \liminf_{r \to 0}r^{2/(p-1)}u(r) 
	\leq \limsup_{r \to 0}r^{2/(p-1)}u(r) = \gamma_2.
	\label{eq1.11}
\end{equation}
These constants $\gamma_1, \gamma_2$ satisfy 
$0 < \gamma_1 \leq L \leq \gamma_2 < A_0$ and 
$\Phi(\gamma_1) = \Phi(\gamma_2)$, where $A_0$ and $\Phi$ are defined by 
$(\ref{eq1.9})$ and $(\ref{eq1.8})$, 
respectively. 
Furthermore, if $\gamma_1 = \gamma_2$ in $(\ref{eq1.11})$, then $u(r) \equiv U_L(r)$.
\end{theorem}

\begin{remark}
It was shown by \cite{Qui} that 
(\ref{eq1.2}) has no singular solution $u \not\equiv U_L$ 
such that $u -U_L$ has finite number of sign changes when $p = (N+2)/(N-2)$. 
Consequently, any singular solution $u \not\equiv U_L$ of (\ref{eq1.2}) intersects $U_L$ infinitely many times. 
This intersection property can also be derived from Theorem \ref{thm1.2}. 
In fact, Theorem \ref{thm1.2} says that any singular positive solution $u \not\equiv U_L$ 
satisfies $\gamma_1 > \gamma_2$ in (\ref{eq1.11}), 
and hence the solution $u$ intersects $U_L$ infinitely many times. 
\end{remark}

For $\alpha > 0$, we denote by $u_{\alpha}$ a unique solution of (\ref{eq1.2}) satisfying 
$u(0) = \alpha$ and $u'(0) = 0$. 
It was shown in \cite[Proposition 3.4 and Theorem 5]{HaWe} that the limit
$$
	\ell(\alpha) = \lim_{r \to \infty}r^{2/(p-1)}u_{\alpha}(r)
$$
exists and is finite. 
For each $\ell > 0$, we denote by $S_{\ell}$ the set of positive solutions 
$u \in C^2(0, \infty)\cap C^1[0, \infty)$ of (\ref{eq1.2}) such that 
$u'(0) = 0$ and $\lim_{r \to \infty}r^{2/(p-1)}u(r) = \ell$.
Define 
\begin{equation}
	\ell^* = \sup\{\ell > 0: S_{\ell} \neq \emptyset\}.
	\label{eq1.14}
\end{equation}
It was shown by \cite[Lemma 3.1]{Naib} that $0 < \ell^* < \infty$ and 
$S_{\ell} \neq \emptyset$ for $0 < \ell < \ell^*$ if $p > (N+2)/N$. 
Let $u$ be a singular positive solution of (\ref{eq1.2}). 
Then, by a slight modification of the argument in \cite{HaWe}, 
we find that the limit
\begin{equation}
	\lim_{r \to \infty}r^{2/(p-1)}u(r) = \ell \geq 0
	\label{eq1.15}
\end{equation}
exists and is finite. We obtain the following result.

\begin{theorem}
\label{thm1.3}
Let $N/(N-2) \leq p \leq (N+2)/(N-2)$, and let 
$u$ be a singular positive solution of $(\ref{eq1.2})$. 
Then the limit $\ell$ in $(\ref{eq1.15})$ satisfies $0 \leq \ell \leq \ell^*$, 
where $\ell^*$ is defined by $(\ref{eq1.14})$.
\end{theorem}

\begin{remark}
(i) The property $\ell \leq \ell^*$ was 
already obtained by \cite[Lemma 7.1]{Naic} in the study of the Cauchy problem for (\ref{eq1.1}). 
In this paper, we will show this result by using the comparison principle for the equation (\ref{eq1.2}). 
The proof is completely different from that of \cite{Naic}, and it might be of interest. 

(ii) It is an interesting open question whether (\ref{eq1.2}) has a singular solution 
$u$ satisfying (\ref{eq1.15}) for any $\ell \in (0, \ell^*]$.
\end{remark}

It is well known that self-similar solutions are closely related to 
nonuniqueness results for the Cauchy problem
\begin{equation}
	\left\{
	\begin{array}{c}
	\partial_t w = \Delta w + w^p, \quad x \in \R^N, \ t > 0, 
	\\[1ex]
	w(x, 0) = w_0(x) \geq 0, \quad x \in \R^N,
	\end{array}
	\right.
	\label{eq1.16}
\end{equation}
where $p > 1$ and $u_0 \in L^q(\R^N)$ with some $q \geq 1$. 
When $(N+2)/N < p < (N+2)/(N-2)$, 
it was shown by Haraux-Weissler \cite{HaWe} that, if $1 \leq q < N(p-1)/2$, 
uniqueness does not hold 
in the class $C([0, T], L^q(\R^N)$ for the initial data $u_0 \equiv 0$. 
More precisely, they proved the existence of a positive self-similar solution $w$ 
satisfying $\|w(\cdot, t)\|_{L^q(\R^N)} \to 0$ as $t \to 0$. 

Let us consider the so-called doubly critical case $q = p = N(P-1)/2$, that is, $q = p = N/(N-2)$. 
Ni and Sacks \cite{NiSa} studied the Cauchy problem 
\begin{equation}
	\left\{
	\begin{array}{ll}
	\partial_t w = \Delta w + w^{N/(N-2)}, & \quad x \in B, \ t > 0, 
	\\[1ex]
	w = 0, & \quad x \in \partial B, \ t > 0,
	\\[1ex]
	w(x, 0) = w_0(x) \geq 0, & \quad x \in B,
	\end{array}
	\right.
	\label{eq1.17}
\end{equation}
where $B$ is the unit ball in $\R^N$ with $N \geq 3$.  
In \cite{NiSa} they constructed a singular positive stationary solution $w^*$ that belongs 
to $L^{N/(N-2)}(B)$. 
Moreover, it was shown by \cite{Weic} that the problem (\ref{eq1.17}) 
with $w_0 = w^*$ has a positive bounded solution 
$w \in C([0, T], L^{N/(N-2)}(B))\cap L^{\infty}_{\rm loc}((0, T), L^{\infty}(B))$ with some $T > 0$. 
Therefore, uniqueness of positive solution does not hold in the class $C([0, T], L^{N/(N-2)}(B))$.  
For the problem (\ref{eq1.16}) with $p = N/(N-2)$, 
nonuniqueness of positive solutions in $C([0, T], L^{N/(N-2)}(\R^N))$ 
was proved by Terraneo \cite{Ter}. 
Later, infinitely many solutions of the problem 
with one initial data were constructed by Matos and Terraneo \cite{MaTe} and Takahashi \cite{Tak}.

We will show nonuniqueness of positive solutions of (\ref{eq1.16}) with $p = N/(N-2)$ by 
using singular self-similar solutions. 
To this end, we define a uniformly local $L^q$ space as follows: For $1 \leq q < \infty$, 
$$
	L^{q}_{\rm ul}(\R^N) = \left\{
	u \in L^q_{\rm loc}(\R^N): \|u\|_{L^q_{\rm ul}(\R^N)} < \infty \right\}, 
$$
where 
$$
	\|u\|_{L^q_{\rm ul}(\R^N)} = \sup_{y \in \R^N}\left(\int_{|x-y| < 1}|u(x)|^q dx\right)^{1/q}.
$$
Let ${\cal L}^q_{\rm ul}(\R^N)$ denote 
the closure of the space of bounded uniformly continuous functions $BUC(\R^N)$ in the space 
$L^q_{\rm ul}(\R^N)$, i.e., 
$$
	{\cal L}^q_{\rm ul}(\R^N) = \overline{BUC(\R^N)}^{\|\cdot\|_{L^q_{\rm ul}(\R^N)}}.
$$
Note here that, if $u \in {\cal L}^q_{\rm ul}(\R^N)$, we have 
\begin{equation}
	\lim_{|y| \to 0}\|u(\cdot+ y)-u(\cdot)\|_{L^q_{\rm ul}(\R^N)} = 0.
	\label{eq1.18}
\end{equation}
(See \cite[Proposition 2.2]{MaTera}.)
It was shown by \cite[Proposition 2.3]{FuIo} that, if $q = N(p-1)/2 > 1$, 
for any $w_0 \in {\cal L}^{q}_{\rm loc}(\R^N)$, 
there exists $T > 0$ such that (\ref{eq1.16}) has at least 
one classical solution $w \in C([0, T]; {\cal L}^{q}_{\rm ul}(\R^N))
\cap L^{\infty}_{\rm loc}((0, T); L^{\infty}(\R^N)) \cap C^{2,1}(\R^N\times (0, T))$. 

Let $p = N/(N-2)$, and let $w^*$ be a singular self-similar solution defined by (\ref{eq1.6}). 
Then $w^*$ satisfies (\ref{eq1.7}) and 
$$
	|x|^{N-2}w^*(x, t) = (|x|/\sqrt{t})^{N-2}u^*(|x|/ \sqrt{t}) \to \ell
	\quad \mbox{as} \ |x| \to \infty.
$$ 
Then $w^*(\cdot, t) \in {\cal L}^{N/(N-2)}_{\rm ul}(\R^N)$ for each fixed $t > 0$. 
Due to the property (\ref{eq1.18}), we see that $w^*(\cdot, t)$ is continuous in 
${\cal L}^{N/(N-2)}_{\rm ul}(\R^N)$ for $t > 0$. 
We consider the problem (\ref{eq1.16}) with $w_0(x) = w^*(x, t_0)$ for fixed $t_0 > 0$. 
Then $w(x, t) = w^*(x, t+t_0)$ satisfies 
$w \in C([0, \infty), {\cal L}^{N/(N-2)}_{\rm ul}(\R^N))$ and 
solves the problem (\ref{eq1.16}) for $t \geq 0$.  
Applying \cite[Proposition 2.3]{FuIo} with $q = N/(N-2)$, we see that 
the problem (\ref{eq1.16}) has at least one classical solution 
$w \in C([0, T]; {\cal L}^{N/(N-2)}_{\rm ul}(\R^N)) 
\cap L^{\infty}_{\rm loc}((0, T); L^{\infty}(\R^N))$ with some $T > 0$.  
Thus we obtain the following.

\begin{corollary}
\label{cor1.4}
Let $p = N/(N-2)$, and let $u^*$ be a singular positive solution of $(\ref{eq1.2})$. 
Define $w^*$ by $(\ref{eq1.6})$, and take any $t_0 > 0$. 
Then $(\ref{eq1.16})$ with $p = N/(N-2)$ and $w_0(x) = w^*(x, t_0)$ has at least two solutions 
in $C([0, T]; {\cal L}^{N/(N-2)}_{\rm ul}(\R^N))$ with some $T > 0$. 
Specifically, one solution is $w(x, t) = w^*(x, t+t_0)$ which is singular at $x = 0$ for all $t \geq 0$, 
and the other solution satisfies $(\ref{eq1.16})$ in the classical sense for 
$(x, t) \in \R^N \times (0, T)$. 
\end{corollary}

The paper is organized as follows. 
We introduce some transformations of the equation (\ref{eq1.2}) in Section 2, 
and we show the existence of infinitely many singular solutions in Section 3.
We give some preliminary lemmas for 
the study of the properties of singular solutions in Section 4, and  
we investigate the asymptotic behavior of singular solutions near the origin 
and give the proof of Theorems \ref{thm1.1} and \ref{thm1.2} in Section 5.
Finally, we show the bound of singular solutions at infinity and prove 
Theorem \ref{thm1.3} in Section 6. 


\section{Transformation of the equation}

In this section, we introduce some transformations of the equation (\ref{eq1.2}). 
Let $N/(N-2) < p \leq (N+2)/(N-2)$, and 
let $u$ be a positive solution of (\ref{eq1.2}) for $0 < r \leq r_0$. 
Define
\begin{equation}
	w(t) = r^{2/(p-1)}u(r) \quad \mbox{with} \ t = -\log r.
	\label{eq2.1}
\end{equation}
Then $w$ satisfies 
\begin{equation}
	w'' -a(t)w' -L^{p-1}w +  w^p = 0 \quad \mbox{for} \ t \geq t_0,
	\label{eq2.2}
\end{equation}
where $t_0 = -\log r_0$, 
\begin{equation}
	a(t) = N-2-\frac{4}{p-1} - \frac{1}{2}e^{-2t}
	\label{eq2.3}
\end{equation}
and $L$ is the constant defined in (\ref{eq1.3}). 
Note that $a(t) \leq 0$ for all $t \in \R$ if $p \leq (N+2)/(N-2)$.
For a solution $w$ of (\ref{eq2.2}), define 
\begin{equation}
	E(w)(t) = \frac{1}{2}w'(t)^2 + \Phi(w(t))
	\quad \mbox{for} \ t \geq t_0,
	\label{eq2.4}
\end{equation}
where $\Phi$ is defined by (\ref{eq1.8}). 
Then we have  
\begin{equation}
	E(w)'(t) = (w''(t) -L^{p-1}w + w^{p})w'(t) = a(t)w'(t)^2 \leq 0 \quad \mbox{for} \ t > t_0,
	\label{eq2.5}
\end{equation}
which implies that $E(w)(t)$ is nonincreasing for $t \geq t_0$.

Next, let us consider the case $p = N/(N-2)$.
Let $u(r)$ be a positive solution of (\ref{eq1.2}) for $0 < r \leq r_0$ with $r_0 < 1$,  
and define $w(t)$ by 
\begin{equation}
	w(t) = r^{N-2}t^{(N-2)/2}u(r) \quad \mbox{with} \ t = -\log r.
	\label{eq2.6}
\end{equation}
Then $w$ satisfies 
\begin{equation}
	tw'' +b(t)w' +B(t)w +  w^p = 0 \quad \mbox{for} \ t \geq t_0,
	\label{eq2.7}
\end{equation}
where $t_0 = -\log r_0$, 
\begin{equation}
	b(t) = (N-2)(t-1) -\frac{1}{2}te^{-2t}
	\label{eq2.8}
\end{equation}
and 
\begin{equation}
	B(t) = -\frac{(N-2)^2}{2} + \frac{N(N-2)}{4t} + \frac{N-2}{4}e^{-2t}.
	\label{eq2.9}
\end{equation}
For a solution $w$ of (\ref{eq2.7}), define 
\begin{equation}
	{\cal E}(w)(t) = \frac{1}{2}tw'(t)^2 + \frac{1}{2}B(t)w(t)^2 + \frac{1}{p+1}w(t)^{p+1}
	\quad \mbox{for} \ t \geq t_0,
	\label{eq2.10}
\end{equation}
Then we have 
\begin{equation}
	{\cal E}(w)'(t) = \left(\frac{1}{2}-b(t)\right)w'(t)^2 
	-\left(\frac{N(N-2)}{8t^2} + \frac{N-2}{4}e^{-2t}\right)w(t)^2
	\quad \mbox{for} \ t > t_0.
	\label{eq2.11}
\end{equation}
Since $te^{-2t} \leq \frac{1}{2e} < \frac{1}{2}$ for $t \geq 0$ and $N \geq 3$, we have
$$
	\frac{1}{2} -b(t) < \frac{1}{2} - (N-2)(t-1) + \frac{1}{4} < 0 \quad \mbox{for} \ t \geq 2.
$$
Then ${\cal E}(w)(t)$ is decreasing for $t \geq \max\{2, t_0\}$.
Define
\begin{equation}
	\Psi(v) = -\frac{(N-2)^2}{4}v^2 + \frac{1}{p+1}v^{p+1}
	\label{eq2.12}
\end{equation}
for $v \geq 0$. 
Then $\Psi(v)$ satisfies
\begin{equation}
	\Psi(v) \geq \Psi(K) = -\frac{1}{2N-2}K^{p+1}, 
	\quad \mbox{where} \ K = \left(\frac{(N-2)^2}{2}\right)^{(N-2)/2}. 
	\label{eq2.13}
\end{equation}
We see that $\Psi(v) < 0$ for $0 < v < B_0$, where 
\begin{equation}
	B_0 = \left(\frac{(N-1)(N-2)}{2}\right)^{(N-2)/N}.
	\label{eq2.14}
\end{equation}
Note that ${\cal E}(w)(t)$ can be written as 
\begin{equation}
	{\cal E}(w)(t) = \frac{1}{2}tw'(t)^2 + \Psi(w(t)) + 
	\left(\frac{N(N-2)}{8t} + \frac{N-2}{8}e^{-2t}\right)w(t)^2
	\quad \mbox{for} \ t \geq t_0
	\label{eq2.15}
\end{equation}
and we have ${\cal E}(w)(t) \geq \Psi(w(t))$ for $t \geq t_0$.

Let $p = N/(N-2)$, and let $u(r)$ be a positive solution of (1.2) for $r > 0$.
Define 
\begin{equation}
	z(s) = r^{2/(p-1)}u(r) = r^{N-2}u(r) \quad \mbox{with} \ s = \log r.
	\label{eq2.16}
\end{equation}
Then $z$ satisfies 
\begin{equation}
	z'' + \left(-(N-2)+ \frac{1}{2}e^{2s}\right)z' +  z^p = 0 \quad \mbox{for} \ s \geq s_0,
	\label{eq2.17}
\end{equation}
where $s_0 = \log r_0$. 
The equation (\ref{eq2.17}) can be written as 
\begin{equation}
	(\eta(s)z')' + \eta(s)z^p = 0, \quad \mbox{where} \quad 
	\eta(s) = \exp\left(-(N-2)s + \frac{1}{4}e^{2s}\right).
	\label{eq2.18}
\end{equation}


\section{Existence of infinitely many singular solutions}

In this section we will show the existence of infinitely many singular positive 
solutions of (\ref{eq1.2}).
For $p \geq (N+2)/N$, define a positive constant 
\begin{equation}
	C_0 = \left(\frac{N}{2}-\frac{1}{p-1}\right)^{1/(p-1)}.
	\label{eq3.1}
\end{equation}
The following result says that (\ref{eq1.2}) has a continuum of singular positive solutions 
in the case $N/(N-2) < p \leq (N+2)/(N-2)$. 

\begin{proposition}
\label{prp3.1}
Let $N/(N-2) < p \leq (N+2)/(N-2)$. 
Choose $r_0 > 0$ so that $r_0^2 > 2/(p-1)$.
Take any $\alpha$ satisfying 
\begin{equation}
	0 < \alpha < \min\{r_0^{-2/(p-1)}C_0, A_0\}
	\label{eq3.2}
\end{equation}
where $A_0$ and $C_0$ are constants defined by $(\ref{eq1.9})$ and $(\ref{eq3.1})$, 
respectively.
If $u$ is a solution of $(\ref{eq1.2})$ satisfying 
\begin{equation}
	u(r_0) = \alpha r_0^{-2/(p-1)} \quad \mbox{and} \quad 
	u'(r_0) = -\frac{2\alpha}{p-1} r_0^{-(p+1)/(p-1)},
	\label{eq3.3}
\end{equation}
then the solution $u(r)$ is positive for all $r > 0$ and satisfies 
$u(r) \to \infty$ as $r \to 0$.
\end{proposition}

In order to prove Proposition \ref{prp3.1}, we need some lemmas.

\begin{lemma}
\label{lem3.2}
Let $w$ be a solution of $(\ref{eq2.2})$ with some $t_0 \in \R$, 
and define $E(w)$ by $(\ref{eq2.4})$. 
Assume that $w(t_0) > 0$ and $E(w)(t_0) < 0$. 
Then $w(t) \geq C$ for all $t \geq t_0$ with some constant $C > 0$. 
\end{lemma}

\begin{proof}
From (\ref{eq2.5}) we see that $E(w)(t)$ is nonincreasing for $t \geq t_0$. 
Then $E(w)(t) \leq E(w)(t_0)$ for $t \geq t_0$. 
In particular, we obtain $\Phi(w(t)) \leq E(w)(t_0) < 0$ for $t \geq t_0$. 
Since $\Phi(0) = \Phi(A) = 0$ and $\Phi(v) < 0$ for $0 < v < A_0$, 
there exists a constant $C > 0$ such that $w(t) \geq C$ for $t \geq t_0$. 
\end{proof}

Let us consider the case $p > (N+2)/N$. 
Multiplying the both sides of (\ref{eq1.2}) by $r^{N-1}$, we obtain 
$$
	(r^{N-1}u')' + \left(\frac{r^Nu}{2}\right)' + r^{N-1}F(u(r)) = 0,
$$
where 
$$
	F(v) = -\left(\frac{N}{2}- \frac{1}{p-1}\right)v + v^p \quad \mbox{for} \ v \geq 0.
$$
Integrating the above on $[r_0, r]$ with $r_0 > 0$, we obtain 
\begin{equation}
	\left(r^{N-1}u'(r) + \frac{r^N u(r)}{2}\right) - 
	\left(r_0^{N-1}u'(r_0) + \frac{r_0^N u(r_0)}{2}\right) 
	+ \int^r_{r_0}s^{N-1}F(u(s))ds = 0.
	\label{eq3.4}
\end{equation}
Note that $\frac{N}{2}- \frac{1}{p-1} > 0$ if $p > (N+2)/N$, and that 
$F(v) < 0$ if and only if $0 < v < C_0$, where $C_0$ is defined by (\ref{eq3.1}).

\begin{lemma}
\label{lem3.3}
Let $p > (N+2)/N$. 
Assume that a solution $u$ of $(\ref{eq1.2})$ satisfies  
\begin{equation}
	u'(r_0) < 0, \quad u'(r_0) + \frac{r_0}{2}u(r_0) > 0 
	\quad \mbox{and} \quad 0 < u(r_0) < C_0
	\label{eq3.5}
\end{equation}
with some $r_0 > 0$, where $C_0$ is the constant defined by $(\ref{eq3.1})$. 
Then $u(r)$ exists and satisfies $u(r) > 0$ for all $r \geq r_0$. 
\end{lemma}

\begin{proof}
Assume that $u(r) > 0$ for $r_0 \leq r < r_1$ with some $r_1 \leq \infty$. 
Note that (\ref{eq1.2}) can be written as 
$$
	(\rho(r)u')' + \rho(r)\left(\frac{2}{p-1}u +u^p\right) = 0 \quad \mbox{for} \ r > 0,
$$
where $\rho(r) = r^{N-1}e^{r^2/4}$. 
Then we have $(\rho(r)u')' < 0$ for $r_0 \leq r < r_1$, 
and hence $\rho(r)u'(r)$ is decreasing for $r_0 \leq r < r_1$. 
Since $u'(r_0) < 0$, we have $u'(r) < 0$ for $r_0 \leq r < r_1$. 
We will show that $r_1 = \infty$. 
Assume by contradiction that $r_1 < \infty$ and $u$ satisfies 
$u(r) > 0$ for $r_0 \leq r < r_1$ and $u(r_1) = 0$. 
Letting $r = r_1$ in (\ref{eq3.4}), we have 
\begin{equation}
	r_1^{N-1}u'(r_1) - 
	\left(r_0^{N-1}u'(r_0) + \frac{r_0^N u(r_0)}{2}\right) 
	+ \int^{r_1}_{r_0}s^{N-1}F(u(s))ds = 0.
	\label{eq3.6}
\end{equation}
We have $u'(r) < 0$ and $0 < u(r) \leq u(r_0) < C_0$ for $r_0 \leq r < r_1$, 
and hence $F(u(r)) < 0$ for $r_0 \leq r < r_1$. 
Then, from (\ref{eq3.6}) we have
$$
	\left(r_0^{N-1}u'(r_0) + \frac{r_0^N u(r_0)}{2}\right)
	= r_1^{N-1}u'(r_1) + \int^{r_1}_{r_0}s^{N-1}F(u(s))ds < 0.
$$
This implies that $u'(r_0) + (r_0/2)u(r_0) < 0$, which contradicts (\ref{eq3.5}).  
Thus we obtain $r_1 = \infty$ and $u(r) > 0$ for all $r \geq r_0$. 
\end{proof}

\begin{proof}[Proof of Proposition \ref{prp3.1}] 
Take $r_0 > 0$ and $\alpha > 0$ such that $r_0^2 > 2/(p-1)$ and 
(\ref{eq3.2}) holds.
Define $w(t)$ by (\ref{eq2.1}). 
From (\ref{eq3.3}) we have $w(t_0) = r_0^{2/(p-1)}u(r_0) = \alpha$ and 
$$
	\frac{d}{dt}w(t_0) = 
	-\left(\frac{2}{p-1}r_0^{2/(p-1)}u(r_0) + r_0^{(p+1)/(p-1)}\frac{du}{dr}(r_0)\right)
	= 0.
$$
Define $E(w)$ by (\ref{eq2.4}). Since $\alpha < A_0$ from (\ref{eq3.2}), we have 
$$
	E(w)(t_0) = \frac{1}{2}w'(t_0)^2 + \Phi(w(t_0)) = \Phi(\alpha) < 0.
$$
From Lemma \ref{lem3.2} we have $w(t) > C$ for $t \geq t_0$ with some constant $C > 0$, 
which implies that 
\begin{equation}
	u(r) \geq Cr^{-2/(p-1)} \quad \mbox{for} \ 0 < r \leq r_0 = e^{-t_0}
	\quad \mbox{and} \quad u(r) \to \infty \quad \mbox{as} \ r \to 0.
	\label{eq3.7}
\end{equation}
From (\ref{eq3.2}) and (\ref{eq3.3}) we have $0 < u(r_0) < C_0$ and $u'(r_0) < 0$.  
From $r_0^2 > 4/(p-1)$ it follows that 
$$
	u'(r_0) + \frac{r_0}{2}u(r_0) = \alpha\left(-\frac{2}{p-1} + \frac{r_0^2}{2}\right)
	r_0^{(p+1)/(p-1)} > 0.
$$
Lemma \ref{lem3.3} implies that $u(r) > 0$ for $r \geq r_0$. 
Combining this and (\ref{eq3.7}), we have $u(r) > 0$ for all $r > 0$ and $u(r) \to \infty$ as $r \to 0$.
\end{proof}

Next we consider the case $p = N/(N-2)$. 
We first show the following lemma.

\begin{lemma}
\label{lem3.4}
Let $\eta(s)$ be defined in $(\ref{eq2.18})$. 
Then 
$$
	\int^{\infty}_{s_0} \frac{1}{\eta(\sigma)}\int^{\sigma}_{s_0}\eta(\tau)d\tau d\sigma < \infty
	\quad \mbox{for} \ s_0 \in \R.
$$
\end{lemma}

\begin{proof}
Let $s_1 = (1/2)\log(4N)$. 
It suffices to show the case $s_0 < s_1$. 
We have 
$$
	\int_{s_0}^s \eta(\sigma)d\sigma = \int_{s_0}^{s_1} + \int_{s_1}^s \eta(\sigma)d\sigma 
	= C_1 + \int^{s}_{s_1}e^{-N\sigma + 
	\frac{1}{8}e^{2\sigma}}e^{2\sigma + \frac{1}{8}e^{2\sigma}}d\sigma
$$
for $s > s_1$.
Since $e^{-Ns + \frac{1}{8}e^{2s}}$ is increasing for $s \geq s_1$, we have 
$$
	\int^{s}_{s_1}e^{-N\sigma + \frac{1}{8}e^{2\sigma}}e^{2\sigma + \frac{1}{8}e^{2\sigma}}d\sigma
	\leq e^{-Ns + \frac{1}{8}e^{2s}}\int^{s}_{s_1}e^{2\sigma + \frac{1}{8}e^{2\sigma}}d\sigma
	\leq 4 e^{-Ns + \frac{1}{8}e^{2s}}e^{\frac{1}{8}e^{2s}} = 4\eta(s)e^{-2s}
$$
for $s > s_1$. 
Thus we have 
$$
	\int_{s_0}^s \eta(\sigma)d\sigma \leq C_1 + 4\eta(s)e^{-2s} \quad \mbox{for} \ s > s_1.
$$
It follows that 
$$
	\int^{\infty}_{s_0} \frac{1}{\eta(\sigma)}\int^{\sigma}_{s_0}\eta(\tau)d\tau d\sigma 
	\leq \int^{\infty}_{s_0}\left(\frac{C_1}{\eta(\sigma)} + 4e^{-2\sigma}\right) d\sigma  < \infty.
$$
Thus we obtain the conclusion.
\end{proof}

\begin{lemma}
\label{lem3.5}
For any $s_0 \in \R$, there exists a constant $D_0 > 0$ such that, if 
$z$ is a solution of $(\ref{eq2.18})$ satisfying 
$z(s_0) = \alpha$ and $z'(s_0) = 0$ with some $0 < \alpha \leq D_0$, 
then the solution $z$ satisfies 
$z(s) > 0$ for all $s \geq s_0$.
\end{lemma}

\begin{proof}
Define $D_0 > 0$ by 
$$
	D_0 = (2pC_{\eta})^{-1/(p-1)}, \quad 
	\mbox{where} \quad 
	C_{\eta} = \int^{\infty}_{s_0} \frac{1}{\eta(\sigma)}\int^{\sigma}_{s_0}\eta(\tau)d\tau d\sigma.
$$
Let $0 < \alpha \leq D_0$. 
Define the set $Z \subset C[s_0, \infty)$ and the mapping ${\cal F}:Z \to C[s_0, \infty)$ by
$$
	Z = \{z \in C[s_0, \infty): {\textstyle \frac{\alpha}{2}} \leq z(s) 
	\leq \alpha \quad \mbox{for} \ s \geq s_0\}
$$
and 
$$
	{\cal F}z(s) = \alpha - \int^{s}_{s_0} \frac{1}{\eta(\sigma)}\int^{\sigma}_{s_0}\eta(\tau)z(\tau)^p
	d\tau d\sigma \quad \mbox{for} \ s \geq s_0,
$$
respectively. 
For $z \in Z$, we have 
$$
	0 \leq \int^{s}_{s_0} \frac{1}{\eta(\sigma)}\int^{\sigma}_{s_0}\eta(\tau)z(\tau)^p
	d\tau d\sigma \leq 
	\alpha^p \int^{\infty}_{s_0} \frac{1}{\eta(\sigma)}\int^{\sigma}_{s_0}\eta(\tau)
	d\tau d\sigma
	\leq D_0^{p-1}C_{\eta}\alpha < \frac{1}{2}\alpha.
$$
Then we obtain  $\frac{\alpha}{2} \leq {\cal F}(z)(s) \leq \alpha$ for $s \geq s_0$, 
and hence ${\cal F}z \in Z$ for any $z \in Z$. 
For $z_1, z_2 \in Z$, we have 
$$
	\begin{array}{rcl}
	|{\cal F}(z_1)(s)-{\cal F}(z_2)(s)|& \leq & \dsp
	\int^{\infty}_{s_0} \frac{1}{\eta(\sigma)}\int^{\sigma}_{s_0}\eta(\tau)
	|z_1(\tau)^p-z_2(\tau)^p| d\tau d\sigma
	\\[3ex]
	& \leq & \dsp  
	\left(\int^{\infty}_{s_0} \frac{1}{\eta(\sigma)}\int^{\sigma}_{s_0}
	\eta(\tau)p\alpha^{p-1} d\tau d\sigma\right)\sup_{s \geq s_0}|z_1(s)-z_2(s)|
	\\[3ex]
	& \leq & \dsp  
	pC_{\eta}D_0^{p-1}\sup_{s \geq s_0}|z_1(s)-z_2(s)| 
	\leq \frac{1}{2}\sup_{s \geq s_0}|z_1(s)-z_2(s)|
	\end{array}
$$
for $s \geq s_0$. 
Then we obtain 
$\sup_{s \geq s_0}|{\cal F}z_1(s)-{\cal F}z_2(s)| \leq \frac{1}{2}\sup_{s \geq s_0}|z_1(s)-z_2(s)|$. 
By the contractive fixed point theorem, there exists a unique $z \in Z$ such that 
${\cal F}z = z$, i.e., 
$$
	z(s) = \alpha - \int^s_{s_0} \frac{1}{\eta(\sigma)}\int^{\sigma}_{s_0}\eta(\tau)z(\tau)^p
	d\tau d\sigma \quad \mbox{for} \ s \geq s_0.
$$
Then $z$ is a solution of (\ref{eq2.18}) 
satisfying $0 < \frac{1}{2}\alpha \leq z(s) \leq \alpha$ for $s \geq s_0$.
\end{proof}

The next proposition says that (\ref{eq1.2}) has a continuum of singular positive solutions 
in the case $p = N/(N-2)$. 

\begin{proposition}
\label{prp3.6}
Let $p = N/(N-2)$. 
Take any $\alpha_0 > 0$ such that $\Psi(\alpha_0) < 0$, that is $0 < \alpha_0 < B_0$, 
where $B_0$ is the constant defined in $(\ref{eq2.14})$.
Choose $t_0 \geq 2$ so large that 
\begin{equation}
	\frac{(N-2)^2}{8t_0}\alpha_0^2+ \Psi(\alpha_0) +
	\left(\frac{N(N-2)}{8t_0} + \frac{N-2}{8}e^{-2t_0}\right)\alpha_0^2 < 0
	\label{eq3.8}
\end{equation}
and 
\begin{equation}
	t_0^{-(N-2)/2}\alpha_0 \leq D_0,
	\label{eq3.9}
\end{equation}
where $\Psi$ is defined by $(\ref{eq2.12})$ and $D_0$ is the constant in Lemma $3.5$. 
Let $r_0 = e^{-t_0}$. 
If $u$ is a solution of $(\ref{eq1.2})$ satisfying 
\begin{equation}
	u(r_0) = r_0^{-(N-2)}t_0^{-(N-2)/2}\alpha_0 
	\quad \mbox{and} \quad 
	u'(r_0) = -(N-2)r_0^{-(N-1)}t_0^{-(N-2)/2}\alpha_0,
	\label{eq3.10}
\end{equation}
then the solution $u(r)$ is positive for all $r > 0$ and satisfies 
$u(r) \to \infty$ as $r \to 0$.
\end{proposition}

\begin{proof}
Take $\alpha_0 > 0$  and $t_0 \geq 2$ such that (\ref{eq3.8}) and (\ref{eq3.9}) hold.
Define $w(t)$ by (\ref{eq2.6}). 
Then 
$$
	\frac{d}{dt}w(t) = r^{N-2}t^{(N-2)/2}\left(
	-(N-2)u(r) - \frac{(N-2)}{2t}u(r) - ru'(r)\right),
$$
where $t = -\log r$.
From (\ref{eq3.10}) we have $w(t_0) = \alpha_0$ and $w'(t_0) = -\frac{N-2}{2t_0}\alpha_0$.  
Define ${\cal E}(w)$ by (\ref{eq2.10}). 
From (\ref{eq2.15}) and (\ref{eq3.8}) we have 
$$
	{\cal E}(w)(t_0) = \frac{(N-2)^2}{8t_0}\alpha_0^2+ \Psi(\alpha_0) +
	\left(\frac{N(N-2)}{8t_0} + \frac{N-2}{8}e^{-2t_0}\right)\alpha_0^2 < 0.
$$
Since ${\cal E}(w)(t)$ is decreasing for $t \geq t_0 (\geq 2)$, we have  
$$
	0 > {\cal E}(w)(t_0) \geq {\cal E}(w)(t) \geq \Psi(w(t))
	\quad \mbox{for} \ t \geq t_0.
$$
By the similar argument as in the proof of Lemma \ref{lem3.2}, 
we obtain $w(t) > C$ for $t \geq t_0$ with some constant $C > 0$. 
This implies that $u(r) > 0$ for $0 < r \leq r_0$ and $u(r) \to \infty$ as $r \to 0$. 

Define $z(s)$ by (\ref{eq2.16}) for $s \geq s_0 = \log r_0$. 
Then $z$ satisfies (\ref{eq2.17}) and 
$$
	z(s_0) = t_0^{-(N-2)/2}\alpha_0 \leq D_0 
	\quad \mbox{and} \quad z'(s_0) = 0.
$$ 
From Lemma \ref{lem3.5}, we have $z(s) > 0$ for $s \geq s_0$, which implies that 
$u(r) > 0$ for $r \geq r_0$. 
Therefore, we have $u(r) > 0$ for all $r > 0$ and $u(r) \to \infty$ as $r \to 0$.
\end{proof}


\section{Preliminary results}

In this section we give some preliminary lemmas to investigate 
the asymptotic properties of singular solutions near the origin in the next section.
For simplicity, define $f(u)$ by 
\begin{equation}
	f(u) = \frac{1}{p-1}u + u^p \quad \mbox{for} \ u \geq 0.
	\label{eq4.1}
\end{equation}
Then (\ref{eq1.2}) can be written as
\begin{equation}
	(\rho(r)u')' + \rho(r)f(u) = 0 \quad \mbox{for} \ r > 0,
	\label{eq4.2}
\end{equation}
where $\rho(r) = r^{N-1}e^{r^2/4}$. 
First we show the following results.

\begin{lemma}
\label{lem4.1}
Let $p \geq N/(N-2)$, and 
let $u \in C^2(0, \infty)$ be a positive solution of $(\ref{eq1.2})$ for $r > 0$. 
Then $u'(r) \leq 0$ for all $r > 0$ and $\rho(r)u'(r) \to 0$ as $r \to 0$. 
Furthermore, $u$ satisfies
\begin{equation}
	-\rho(r)u'(r) = \int^r_0 \rho(s)f(u(s))ds
	\quad \mbox{for} \ 0 < r \leq r_0.
	\label{eq4.3}
\end{equation}
\end{lemma}

\begin{proof}
From (\ref{eq4.2}) we have $(\rho(r)u'(r))' = -\rho(r)f(u(r)) < 0$, 
and then $\rho(r)u'(r)$ is decreasing for $r > 0$. 
First we will show that $u'(r) \leq 0$ for all $r > 0$. 
Assume by a contradiction that there exists $r_1 > 0$ such that $u'(r_1) > 0$. 
Then we have $\rho(r)u'(r) > \rho(r_1)u'(r_1) = C_1 >0$ for $0 < r \leq r_1$, 
and hence 
$$
	u'(r) >  C_1e^{-r^2/4}r^{-(N-1)} \geq C_1e^{-r_1^2/4}r^{-(N-1)}
	\quad \mbox{for} \ 0 < r \leq r_1.
$$
Integrating the above on $[r, r_1]$ and letting $r \to 0$, we obtain 
$$
	u(r_1)-u(r) > C_1e^{-r_1^2/4}\int^{r_1}_r s^{-(N-1)}ds \to \infty \quad \mbox{as} \ r \to 0,
$$
which implies that $u(r) \to -\infty$ as $r \to 0$. 
This contradicts that $u(r)$ is positive for $r > 0$. 
Thus we obtain $u'(r) \leq 0$ for all $r > 0$.

Next we will show that 
\begin{equation}
	\lim_{r \to 0}\rho(r)u'(r) = 0.
	\label{eq4.4}
\end{equation}
Note that $\rho(r)u'(r)$ is decreasing and $\rho(r)u'(r) \leq 0$  for $r > 0$. 
Assume by contradiction that $\lim_{r \to 0}\rho(r)u'(r) = -c$ for some $c > 0$. 
Then we have $\rho(r)u'(r) < -c$ for $r > 0$.
Take any $r_2 > 0$. Then we have 
$$
	u'(r) < -c e^{-r^2/4}r^{-(N-1)} \leq -c e^{-r_2^2/4}r^{-(N-1)} 
	\quad \mbox{for} \ 0 < r \leq r_2. 
$$
Integrating the above on $[r, r_2]$, we obtain 
$$
	-u(r) < u(r_2)-u(r) 
	\leq -C_2r^{2-N} 
	\quad \mbox{for} \ 0 < r \leq \frac{r_2}{2} 
$$
with some constant $C_2 > 0$.
Then we obtain 
\begin{equation}
	u(r) \geq C_3r^{2-N} \quad \mbox{for} \ 0 < r \leq \frac{r_2}{2}
	\label{eq4.5}
\end{equation}
with $C_3 = 1/C_2 > 0$. 
Integrating (\ref{eq4.2}) on $[r_3, r]$, and letting $r_3 \to 0$, we obtain 
$$
	-r^{N-1}e^{r^2/4}u'(r)  -c = \int^r_0 s^{N-1}e^{s^2/4}f(u(s))ds 
	\geq \int^r_0 s^{N-1}u(s)^pds.
$$
From (\ref{eq4.5}) and $p \geq N/(N-2)$ we have
$$
	\int^r_0 s^{N-1}u(s)^pds \geq C_3^p\int^r_0 s^{N-1-p(N-2)}ds = \infty.
$$
Thus we have a contradiction, and hence (\ref{eq4.4}) holds. 
Integrating (\ref{eq4.2}) on $[r_3, r]$, and letting $r_3 \to 0$, 
we obtain (\ref{eq4.3}).
\end{proof}

\begin{lemma}
\label{lem4.2}
Let $N/(N-2) < p \leq (N+2)/(N-2)$, and let $u$ be a singular positive solution of $(\ref{eq1.2})$. 
Define $w(t)$ by $(\ref{eq2.1})$ for $t \geq t_0$.  
Then $w(t)$, $w'(t)$ and $w''(t)$ are bounded on $[t_0, \infty)$.
\end{lemma}

\begin{proof}
Define $E(w)(t)$ by (\ref{eq2.4}). 
From (\ref{eq2.5}) we see that $E(w)(t)$ is nonincreasing for $t \geq t_0$. 
Then $E(w)(t)$ is bounded for $t \geq t_0$, 
which implies that $w'(t)$ and $w(t)$ is bounded for $t \geq t_0$. 
Note that $a(t)$, defined by (\ref{eq2.3}), is bounded on $[t_0, \infty)$.
From (\ref{eq2.2}), $w''(t)$ is also bounded on $[t_0, \infty)$. 
\end{proof}

Lemma \ref{lem4.2} implies that $r^{2/(p-1)}u(r)$ is bounded near the origin. 
Furthermore, we obtain the following.

\begin{lemma}
\label{lem4.3}
Let $N/(N-2) < p \leq (N+2)/(N-2)$, 
and let $u$ be a singular positive solution of $(\ref{eq1.2})$. 
Then 
\begin{equation}
	\limsup_{r \to 0}r^{2/(p-1)}u(r) > 0.
	\label{eq4.6}
\end{equation}
\end{lemma}

\begin{proof}
Assume by contradiction that 
\begin{equation}
	\lim_{r \to 0}r^{2/(p-1)}u(r) = 0.
	\label{eq4.7}
\end{equation}
First we will show that 
\begin{equation}
	(r^{2/(p-1)}u(r))' > 0 \quad \mbox{for} \ 0 < r \leq r_1
	\label{eq4.8}
\end{equation}
with some $r_1 > 0$. 
Define $w(t)$ by (\ref{eq2.1}). 
Then $w$ satisfies (\ref{eq2.2}). 
From (\ref{eq4.7}) we have $w(t) \to 0$ as $t \to \infty$. 
Then there exists $t_1 \geq t_0$ such that
$$
	-L^{p-1} +  w(t)^{p-1} < 0 \quad \mbox{for} \ t \geq t_1,
$$
where $L > 0$ is the constant defined in (\ref{eq1.3}). 
From (\ref{eq2.2}) we obtain $w''(t) - a(t)w'(t) > 0$ for $t \geq t_1$, and hence 
$$
	(e^{-\int^t_{t_1}a(s)ds}w'(t))' > 0 \quad \mbox{for} \ t \geq t_1. 
$$
This implies that $e^{-\int^t_{t_1}a(s)ds}w'(t)$ is increasing for $t \geq t_1$. 
Then we have either $w'(t) < 0$ for all $t \geq t_1$ or $w'(t) > 0$ for $t \geq t_2$ 
with some $t_2 \geq t_1$. 
Since $w(t) > 0$ and $w(t) \to 0$ as $t \to \infty$, the former case has to hold. 
Thus we obtain $w'(t) < 0$ for all $t \geq t_1$, 
which implies that (\ref{eq4.8}) holds with $r_1 = e^{-t_1}$.

From (\ref{eq4.1}) and (\ref{eq4.7}) we have 
$$
	\frac{r^2f(u(r))}{u(r)} = r^2\left(\frac{1}{p-1} + u(r)^{p-1}\right) 
	=  \frac{r^2}{p-1} + (r^{2/(p-1)}u(r))^{p-1} \to 0 \quad \mbox{as} \ r \to 0.
$$
Then, for $\vep > 0$ to be determined later, there exists $r_2 \in (0, r_1]$ such that
$$
	r^2f(u(r)) < \vep u(r) \quad \mbox{for} \ 0 < r \leq r_2.
$$
From (\ref{eq4.8}) we have
$$
	\begin{array}{rcl}
	\displaystyle 
	\int^r_0 s^{N-1}e^{s^2/4}f(u(s))ds & < &
	\displaystyle 
	\vep e^{r^2/4} \int^r_0 s^{N-3}u(s)ds
	\\[2ex]
	& \leq &
	\displaystyle 
	\vep e^{r^2/4} r^{2/(p-1)}u(r)\int^r_0 s^{N-3-2/(p-1)}ds 
	\\[2ex]
	& = &
	\displaystyle 
	\frac{e^{r^2/4}\vep}{N-2-\frac{2}{p-1}}r^{N-2}u(r)
	\end{array}
$$
for $0 < r \leq r_2$.
It follows from (\ref{eq4.3}) that  
\begin{equation}
	-r^{N-1}e^{r^2/4}u'(r) = \int^r_0 s^{N-1}e^{s^2/4}f(u(s))ds \leq 
	\frac{e^{r^2/4}\vep}{N-2-\frac{2}{p-1}}r^{N-2}u(r)
	\label{eq4.9}
\end{equation}
for $0 < r \leq r_2$.
Define  
$$
	\sigma = \frac{\vep}{N-2-\frac{2}{p-1}},
$$
and take $\vep > 0$ so small that $\sigma < 1$. 
From (\ref{eq4.9}) we have 
$$
	-ru'(r) \leq \sigma u(r) \quad \mbox{for} \ 0 < r \leq r_2,
$$
which implies that $(r^{\sigma}u(r))' \geq 0$ for $0 < r \leq r_2$. 
Then we obtain $r^{\sigma}u(r) \leq r_1^{\sigma}u(r_1)$ for $0 < r \leq r_2$, 
and hence $u(r) = O(r^{-\sigma})$ and $f(u(r)) = O(r^{-\sigma})$ as $r \to 0$. 
From (\ref{eq4.3}) we obtain 
$$
	-r^{N-1}e^{r^2/4}u'(r) 
	\leq e^{r^2/4}\int^r_0s^{N-1}f(u(s))ds = O(e^{r^2/4}r^{N-\sigma})
	\quad \mbox{as} \ r \to 0.
$$
Then we obtain $u'(r) = O(r^{1-\sigma})$ as $r \to 0$. 
Since $\sigma < 1$, we have $u'(r) \to 0$ as $r \to 0$, 
and hence $\lim_{r \to 0}u(r) < \infty$. 
This is a contradiction. 
Thus we obtain (\ref{eq4.6}).
\end{proof}

In the case $p = N/(N-2)$ we obtain the following.

\begin{lemma}
\label{lem4.4}
Let $p = N/(N-2)$, and let $u$ be a singular positive solution of $(\ref{eq1.2})$. 
Define $w(t)$ by $(\ref{eq2.6})$ for $t \geq t_0$ with $t_0 \geq 2$. 
Then $w(t)$, $tw'(t)^2$ and $tw'(t)w''(t)$ are bounded on $[t_0, \infty)$.
\end{lemma}

\begin{proof}
Define ${\cal E}(w)(t)$ by (\ref{eq2.10}). 
Since ${\cal E}(w)(t)$ is decreasing for $t \geq t_0 (\geq 2)$,  
we have ${\cal E}(w)(t) \leq {\cal E}(w)(t_0)$ for $t \geq t_0$. 
Then, from (\ref{eq2.15}), it follows that  
$$
	{\cal E}(w)(t_0) \geq {\cal E}(w)(t) \geq \frac{1}{2}tw'(t)^2 + \Psi(w(t)) 
	\quad \mbox{for} \ t \geq t_0.
$$
This implies that $tw'(t)^2$ and $w(t)$ are bounded for $t \geq t_0$. 
Multiplying (\ref{eq2.7}) by $w'(t)$, we have 
$$
	tw''w' + b(t)(w')^2 + B(t)ww' + w^pw' = 0 
	\quad \mbox{for} \ t \geq t_0,
$$
where $b(t)$ and $B(t)$ are defined by (\ref{eq2.8}) and (\ref{eq2.9}), respectively. 
Since $tw'(t)^2$ is bounded for $t \in [t_0, \infty)$, 
we see that $b(t)w'(t)^2$ is bounded, and hence 
$tw'(t)w''(t)$ is also bounded for $t \in [t_0, \infty)$. 
\end{proof}

Lemma \ref{lem4.4} implies that $r^{2/(p-1)}(-\log(r))^{(N-2)/2}u(r)$ is bounded near the origin. 
Furthermore, we obtain the following.

\begin{lemma}
\label{lem4.5}
Let $p = N/(N-2)$, and let $u$ be a singular positive solution of $(\ref{eq1.2})$. 
Then 
\begin{equation}
	\limsup_{r \to 0}r^{N-2}(-(\log r))^{(N-2)/2)}u(r) > 0.
	\label{eq4.10}
\end{equation}
\end{lemma}

\begin{proof}
Assume by contradiction that 
\begin{equation}
	\limsup_{r \to 0}r^{N-2}(-(\log r))^{(N-2)/2)}u(r) = 0.
	\label{eq4.11}
\end{equation}
First we will show that 
\begin{equation}
	\frac{d}{dr}(r^{N-2}(-(\log r))^{(N-2)/2)}u(r)) > 0
	\quad \mbox{for} \ 0 < r \leq r_1
	\label{eq4.12}
\end{equation}
with some $r_1 > 0$. 
Define $w(t)$ by (\ref{eq2.6}). 
From (\ref{eq4.11}) we have $w(t) \to 0$ as $t \to \infty$. 
Then there exists $t_1 > 0$ such that
$$
	-B(t) +  w(t)^{p-1} < 0 \quad \mbox{for} \ t \geq t_1,
$$
where $B(t)$ is defined by (\ref{eq2.9}). 
From (\ref{eq2.7}) we obtain $tw''(t) + b(t)w'(t) > 0$ for $t \geq t_1$, and hence 
$$
	(e^{\int^t_{t_1}\frac{b(s)}{s}ds}w'(t))' > 0 \quad \mbox{for} \ t \geq t_1. 
$$
This implies that $e^{\int^t_{t_1}\frac{b(s)}{s}ds}w'(t)$ is increasing for $t \geq t_1$. 
Then we have either $w'(t) < 0$ for all $t \geq t_1$ or $w'(t) > 0$ for $t \geq t_2$ 
with some $t_2 \geq t_1$. 
Since $w(t) > 0$ and $w(t) \to 0$ as $t \to \infty$, the former case has to hold. 
Then we obtain  $w'(t) < 0$ for all $t \geq t_1$, 
which implies that (\ref{eq4.12}) holds with $r_1 = e^{-t_1}$.

From (\ref{eq4.11}) we have 
$$
	r^2(-\log r)u(r)^{p-1} = (r^{N-2}(-(\log r))^{(N-2)/2}u(r))^{2/(N-2)} 
	\to 0 \quad \mbox{as} \ r \to 0.
$$
Then it follows from (\ref{eq4.1}) that 
$$
	\frac{r^2(-\log r)f(u(r))}{u(r)} =
	r^2(-\log r)\left(\frac{1}{p-1} + u(r)^{p-1}\right) 
	\to 0 \quad \mbox{as} \ r \to 0.
$$
Then, for $\vep > 0$ to be determined later, there exists $r_2 \in (0, r_1]$ such that
$$
	r^2(-(\log r))f(u(r)) < \vep u(r) \quad \mbox{for} \ 0 < r \leq r_2.
$$
From (\ref{eq4.12}) we have
$$
	\begin{array}{rcl}
	\displaystyle 
	\int^r_0 s^{N-1}e^{s^2/4}f(u(s))ds & \leq &
	\displaystyle 
	\vep e^{r^2/4} \int^r_0 s^{N-3}(-\log r)^{-1}u(s)ds
	\\[2ex]
	& \leq &
	\displaystyle 
	\vep e^{r^2/4} r^{N-2}(-(\log r))^{(N-2)/2}u(r)
	\int^r_0 s^{-1}(-\log r)^{-N/2}ds 
	\\[2ex]
	& \leq &
	\displaystyle 
	\frac{2e^{r^2/4}\vep}{N-2}r^{N-2}u(r)
	\end{array}
$$
for $0 < r \leq r_2$.
It follows from (\ref{eq4.3}) that  
\begin{equation}
	-r^{N-1}e^{r^2/4}u'(r) = \int^r_0 s^{N-1}e^{s^2/4}f(u(s))ds \leq 
	\frac{2e^{r^2/4}\vep}{N-2}r^{N-2}u(r)
	\label{eq4.13}
\end{equation}
for $0 < r \leq r_2$.
Let $\sigma = 2\vep/(N-2)$, and take $\vep > 0$ so small that $\sigma < 1$. 
From (\ref{eq4.13}) we have $-ru'(r) \leq \sigma u(r)$ for $0 < r \leq r_2$, 
which implies that $(r^{\sigma}u(r))' \geq 0$ for $0 < r \leq r_2$.
Then, by the same argument as in the proof of Lemma \ref{lem4.3}, we obtain 
$\lim_{r \to 0}u(r) < \infty$, 
which is a contradiction. 
Thus we obtain (\ref{eq4.10}).
\end{proof}


\section{Asymptotic behavior of the solutions near the origin}

In this section, we consider the 
asymptotic behavior of the singular solutions near the origin.
First we consider the case $N/(N-2) < p \leq (N+2)/(N-2)$.

\begin{proposition}
\label{prp5.1}
Let $N/(N-2) < p \leq (N+2)/(N-2)$, and let $u$ be a singular positive solution of $(\ref{eq1.2})$. 
Define $w(t)$ by $(\ref{eq2.1})$. 

{\rm (i)} Let $N/(N-2) < p < (N+2)/(N-2)$. Then $w(t)$ satisfies 
\begin{equation}
	\lim_{t \to \infty}w(t) = L \quad \mbox{and} \quad 
	\lim_{t \to \infty}w'(t) = 0,
	\label{eq5.1}
\end{equation}
where $L$ is the constant defined in $(\ref{eq1.3})$. 

{\rm (ii)} 
Let $p = (N+2)/(N-2)$. 
Then there exist constants $\gamma_1, \gamma_2$ 
such that 
\begin{equation}
	\gamma_1 = \liminf_{t \to \infty}w(t) \leq \limsup_{t \to \infty}w(t) = \gamma_2. 
	\label{eq5.2}
\end{equation}
The constants $\gamma_1, \gamma_2$ satisfy 
\begin{equation}
	0 < \gamma_1 \leq L \leq \gamma_2 < A_0 \quad \mbox{and} \quad 
	\Phi(\gamma_1) = \Phi(\gamma_2),
	\label{eq5.3}
\end{equation}
where $A_0$ and  $\Phi$ are defined by $(\ref{eq1.9})$ and $(\ref{eq1.8})$, respectively. 
If $\gamma_1 = \gamma_2$, then $w(t) \equiv L$.
\end{proposition}

To prove Proposition \ref{prp5.1}, we need a series of lemmas.

\begin{lemma}
\label{lem5.2}
Let $w$ be defined as in Proposition $\ref{prp5.1}$.  
Define $E(w)(t)$ by $(\ref{eq2.4})$. 
Then $w$ satisfies the following {\rm (i)} and {\rm (ii)}.

\begin{itemize}
\item[{\rm (i)}] $\lim_{t \to \infty}E(w)(t) = \zeta$ for some $\zeta \geq \Phi(L)$. 

\item[{\rm (ii)}] One has  
\begin{equation}
	\int^{\infty}_{t_0}a(s)w'(s)^2ds > -\infty,
	\label{eq5.4}
\end{equation}
where $a \leq 0$ is defined by $(\ref{eq2.3})$.
\end{itemize}
\end{lemma}

\begin{proof}
From (\ref{eq1.10}) and (\ref{eq2.4}) we have 
$$
	E(w)(t) \geq \Phi(w(t)) \geq \Phi(L) 
	\quad  \mbox{for} \ t \geq t_0.
$$ 
From (\ref{eq2.5}) we see that $E(w)(t)$ is nonincreasing for $t \geq t_0$. 
Then there exists a limit 
$\zeta = \lim_{t \to \infty}E(w)(t) \geq \Phi(L)$.
Thus (i) holds.

Integrating (\ref{eq2.5}) on $[t_0, t]$, and letting $t \to \infty$, we have 
$$
	\zeta- E(w)(t_0) =  \int^{\infty}_{t_0}a(s)w'(s)^2ds > -\infty.
$$
Thus (\ref{eq5.4}) holds.
\end{proof}

\begin{lemma}
\label{lem5.3}
Let $w$ be defined as in Proposition $\ref{prp5.1}$. 
If $w$ satisfies 
\begin{equation}
	\lim_{t \to \infty}w(t) = \gamma
	\label{eq5.5}
\end{equation}
for some $\gamma > 0$, then $(\ref{eq5.1})$ holds.
\end{lemma}

\begin{proof}
First, we will show that 
\begin{equation}
	\lim_{t \to \infty}w'(t) = 0. 
	\label{eq5.6}
\end{equation}
Define $E(w)(t)$ by (\ref{eq2.4}). 
From (\ref{eq5.5}) and Lemma \ref{lem5.2} (i), we have 
\begin{equation}
	\lim_{t \to \infty}\frac{w'(t)^2}{2} = 
	\lim_{t \to \infty}(E(w)(t) - \Phi(w(t))) = \zeta - \Phi(\gamma).
	\label{eq5.7}
\end{equation}
Then it suffices to show that $\zeta = \Phi(\gamma)$. 
Since $E(w) \geq \Phi(w)$, we have $\zeta \geq \Phi(\gamma)$. 
Assume that $\zeta > \Phi(\gamma)$. 
Then, from (\ref{eq5.7}), there exists $t_1 \geq t_0$ such that   
$$
	|w'(t)| > \frac{\sqrt{2(\zeta - \Phi(\gamma))}}{2} > 0 \quad \mbox{for} \ t \geq t_1,
$$
which implies that $\lim_{t \to \infty}|w(t)| = \infty$. 
This is a contradiction. Thus we obtain $\zeta = \Phi(\gamma)$, 
and hence (\ref{eq5.6}) holds.

Next we will show that $\gamma = L$. 
Assume to the contrary that $\gamma \neq L$ in (\ref{eq5.5}). 
Letting $t \to \infty$ in (\ref{eq2.2}), from (\ref{eq5.5}) and (\ref{eq5.6}) we obtain  
$$
	\lim_{t \to \infty}w''(t) = -L^{p-1}\gamma + \gamma^{p} \neq 0.
$$
Then we obtain 
$$
	|w''(t)| > \frac{1}{2}|-L^{p-1}\gamma + \gamma^{p}| 
	\quad \mbox{for} \ t \geq t_2 
$$
with some $t_2 \geq t_0$, 
which implies that $|w'(t)| \to \infty$ as $t \to \infty$. 
This is a contradiction. 
Thus we obtain $\gamma = L$ in (\ref{eq5.5}). 
As a consequence, (\ref{eq5.1}) holds.
\end{proof}

\begin{lemma}
\label{lem5.4}
Let $w$ be defined as in Proposition $\ref{prp5.1}$. 
Assume that the limit of $w(t)$ as $t \to \infty$ does not exist. 
Then the following {\rm (i)--(iii)} hold.
\begin{itemize}
\item[{\rm (i)}]
There exist constants $\gamma_1 < \gamma_2$ satisfying $(\ref{eq5.2})$ and 
\begin{equation}
	0 < \gamma_1 < L < \gamma_2 < A_0 \quad \mbox{and} \quad 
	\Phi(\gamma_1) = \Phi(\gamma_2) = \zeta,
	\label{eq5.8}
\end{equation}
where $\zeta$ is the constant in Lemma $\ref{lem5.2}$ {\rm (i)} 
and $\Phi$ is defined by $(\ref{eq1.8})$.

\item[{\rm (ii)}] 
One has $\Phi(L) < \zeta < 0$. 

\item[{\rm (iii)}]
There exists an infinite sequence $\tau_n \to \infty$ such that 
$w(\tau_n) = L$ for $n = 1, 2, \ldots$.
\end{itemize}
\end{lemma}

\begin{proof}
Since $w(t)$ is positive and bounded for $t \geq t_0$ by Lemma \ref{lem4.2}, 
there exist $0 \leq \gamma_1 < \gamma_2$ such that  
$$
	\liminf_{t \to \infty}w(t) = \gamma_1 
	\quad \mbox{and} \quad 
	\limsup_{t \to \infty}w(t) = \gamma_2.
$$
Then there exist sequences $t_n \to \infty$ and $s_n \to \infty$ such that 
\begin{equation}
	w'(t_n) = w'(s_n) = 0, \quad \lim_{n \to \infty}w(t_n) = \gamma_1, 
	\quad \mbox{and} \quad 
	\lim_{n \to \infty}w(s_n) = \gamma_2.
	\label{eq5.9}
\end{equation}
Define $E(w)(t)$ by (\ref{eq2.4}). Then we obtain 
$$
	\lim_{n \to \infty}E(w)(t_n) = \lim_{n \to \infty}\Phi(w(t_n)) = \Phi(\gamma_1) 
	\quad \mbox{and} \quad  
	\lim_{n \to \infty}E(w)(s_n) = \lim_{n \to \infty}\Phi(w(s_n)) = \Phi(\gamma_2). 
$$
By Lemma \ref{lem5.2} (i), we have 
\begin{equation}
	\Phi(\gamma_1) = \Phi(\gamma_2) = \zeta = \lim_{t \to \infty}E(w)(t).
	\label{eq5.10}
\end{equation}
Since $0 \leq \gamma_1 < \gamma_2$ and $\Phi(v)$ is decreasing for $0 < v < L$ and 
increasing for $v > L$, we obtain  
\begin{equation}
	\gamma_1 < L < \gamma_2 \quad \mbox{and} \quad 
	\Phi(L) < \zeta = \Phi(\gamma_1) = \Phi(\gamma_2).
	\label{eq5.11}
\end{equation}
Then, for sufficiently large $n \in \N$, we have $0 < w(t_n) < L$, and then 
$E(w)(t_n) = \Phi(w(t_n)) < 0$. 
Since $E(w)(t)$ is nonincreasing, we obtain $\zeta < 0$ in (\ref{eq5.10}), 
which implies that $0 < \gamma_1 < \gamma_2 < A_0$. Thus (\ref{eq5.8}) holds.
From (\ref{eq5.9}) and (\ref{eq5.11}), there exists an infinite sequence 
$\tau_n \to \infty$ such that $w(\tau_n) = L$ for $n = 1, 2, \ldots$.
Thus (i)--(iii) hold.
\end{proof}

We are now in a position to prove Proposition \ref{prp5.1}.

\begin{proof}[Proof of Proposition \ref{prp5.1}]
(i) Let $N/(N-2) < p < (N+2)/(N-2)$. 
First we will show that (\ref{eq5.5}) holds with  some $\gamma > 0$. 
Assume by contradiction that the limit of $w(t)$ as $t \to \infty$ does not exist. 
Define $E(w)(t)$ by (\ref{eq2.4}). 
By Lemma \ref{lem5.2} (i) we have $E(w)(t) \to \zeta$ with some $\zeta \geq \Phi(L)$.
Lemma \ref{lem5.4} implies that $\zeta > \Phi(L)$ and 
there exists a sequence $\tau_n \to \infty$ such that 
$w(\tau_n) = L$ for $n = 1, 2, \ldots$.
Then we obtain 
$$
	\lim_{n \to \infty}E(w)(\tau_n) = 
	\lim_{n \to \infty}\left(\frac{w'(\tau_n)^2}{2} + \Phi(L)\right) = \zeta.
$$
It follows that 
$$
	\lim_{n \to \infty}\frac{w'(\tau_n)^2}{2} = \zeta - \Phi(L) > 0.
$$
Hence, there exists an integer $n_0$ such that 
$$
	w'(\tau_n)^2 \geq \zeta-\Phi(L) \quad \mbox{for} \ n \geq n_0.
$$
Since $w''(t)$ is bounded for $t \geq t_0$ by Lemma \ref{lem4.2}, 
there exists $\rho > 0$ such that 
$$
	w'(t)^2 \geq \frac{\zeta-\Phi(L)}{2} 
	\quad \mbox{for} \ \tau_n-\rho \leq t \leq \tau_n+ \rho
	\quad \mbox{with} \ n \geq n_0.
$$
This implies that 
$$
	\int^{\infty}_{t_0}w'(t)^2dt = \infty.
$$
Since $p < (N+2)/(N-2)$, we have $a(t) \to N-2-\frac{4}{p-1} < 0$ as 
$t \to \infty$. Thus we obtain 
$$
	\int^{\infty}_{t_0}a(t)w'(t)^2dt = -\infty,
$$
which contradicts Lemma \ref{lem5.2} (ii). 
Thus (\ref{eq5.5}) holds with some $\gamma \geq 0$. 
Lemma \ref{lem4.3} implies that $\gamma > 0$. 
By Lemma \ref{lem5.3}, we obtain (\ref{eq5.1}).

(ii) Let $p = (N+2)/(N-2)$. 
Assume first that the limit of $w(t)$ as $t \to \infty$ does not exist. 
Then, by Lemma \ref{lem5.4}, there exist constants 
$\gamma_1 < \gamma_2$ satisfying (\ref{eq5.2}) and (\ref{eq5.3}).
Assume now that the limit of $w(t)$ as $t \to \infty$ exists, that is, 
$\gamma_1 = \gamma_2 = \gamma$ in (\ref{eq5.2}).
Then, by Lemma \ref{lem4.3}, we have $\gamma > 0$. 
By Lemma \ref{lem5.3}, we have (\ref{eq5.1}), and hence $\gamma = L$. 
We will show that $w(t) \equiv L$. 
Integrating (\ref{eq2.5}) on $[t, \tau]$ with $t \geq t_0$, we have 
\begin{equation}
	E(w)(\tau)- E(w)(t) = \int^{\tau}_{t}a(s)w'(s)^2ds.
	\label{eq5.12}
\end{equation}
From (\ref{eq5.1}) we have
$$
	\lim_{\tau \to \infty}E(w)(\tau) = \lim_{\tau \to \infty}
	\left(\frac{1}{2}w'(\tau)^2 + \Phi(w(\tau))\right) = \Phi(L).
$$ 
Then, letting $\tau \to \infty$ in (\ref{eq5.12}), we obtain 
\begin{equation}
	\Phi(L) - E(w(t)) = \int^{\infty}_t a(s)w'(s)^2 ds
	\quad \mbox{for} \ t \geq t_0.
	\label{eq5.13}
\end{equation}
By the definition of $E(w)(t)$ and the property (\ref{eq1.10}), we have 
\begin{equation}
	E(w(t))-\Phi(L) = \frac{1}{2}w'(t)^2 + \Phi(w(t)) -\Phi(L) \geq \frac{1}{2}w'(t)^2. 
	\label{eq5.14}
\end{equation}
Then, from (\ref{eq5.13}) and (\ref{eq5.14}) we obtain  
\begin{equation}
	\frac{1}{2}w'(t)^2 \leq -\int^{\infty}_t a(s)w'(s)^2 ds \quad \mbox{for} \ t \geq t_0.
	\label{eq5.15}
\end{equation}
Note here that $a(t) = -e^{-2t}/2 < 0$ when $p = (N+2)/(N-2)$. 
Define 
$$
	W(t) = \int^{\infty}_t a(s)w'(s)^2 ds \leq 0 
	\quad \mbox{for} \ t \geq t_0.
$$
Then, from (\ref{eq5.15}) we have 
$$
	W'(t) = -a(t)w'(t)^2 \leq  2a(t)W(t) \quad \mbox{for} \ t \geq t_0,
$$ 
which implies that 
$$
	\frac{d}{dt}\left(e^{-2\int^t_{t_0}a(s)ds} W(t)\right) \leq 0 \quad \mbox{for} \ t \geq t_0.
$$
By Lemma \ref{lem5.2} (ii) we have $W(t) \to 0$ as $t \to \infty$.  
Then, integrating the above on $[t, \tau]$, and letting $\tau \to \infty$, 
we obtain $W(t) \geq 0$ for $t \geq t_0$. 
This implies that $W(t) \equiv 0$, and hence $w'(t) \equiv 0$. 
From (\ref{eq5.1}) we have $w(t) \equiv L$. 
\end{proof}

In the case $p = N/(N-2)$, we obtain the following.

\begin{proposition}
\label{prp5.5}
Let $p = N/(N-2)$, and let $u$ be a singular positive solution of $(1.2)$. 
Define $w(t)$ by $(2.6)$ for $t \geq t_0$.
Then 
\begin{equation}
	\lim_{t \to \infty}w(t) = 
	\left(\frac{(N-2)^2}{2}\right)^{(N-2)/2}.
	\label{eq5.16}
\end{equation}
\end{proposition}

To prove Proposition \ref{prp5.5}, we need a series of lemmas.

\begin{lemma}
\label{lem5.6}
Let $w$ be defined as in Proposition $\ref{prp5.5}$, 
and define ${\cal E}(w)(t)$ by $(\ref{eq2.10})$. 
Then $w$ satisfies the following {\rm (i)} and {\rm (ii)}.

\begin{itemize}
\item[{\rm (i)}] $\lim_{t \to \infty}{\cal E}(w)(t) = \zeta$ for some $\zeta \geq \Psi(K)$, 
where $K$ is defined in $(\ref{eq2.13})$. 

\item[{\rm (ii)}] One has  
$$
	\int^{\infty}_{t_0}sw'(s)^2ds < \infty.
$$
\end{itemize}
\end{lemma}

\begin{proof}
From (\ref{eq2.13}) and (\ref{eq2.15}) we have 
$$
	{\cal E}(w)(t) \geq \Psi(w(t)) \geq \Psi(K) 
	\quad  \mbox{for} \ t \geq t_0.
$$ 
Recall that 
${\cal E}(w)(t)$ is nonincreasing for $t \geq \max\{2, t_0\}$. 
Then there exists a limit 
$\zeta = \lim_{t \to \infty}{\cal E}(w)(t) \geq \Psi(K)$.
Thus (i) holds.

Integrating (\ref{eq2.11}) on $[t_0, t]$, we have 
$$
	{\cal E}(w)(t) -{\cal E}(w)(t_0) 
	= \int^t_{t_0}\left(\frac{1}{2}-b(s)\right)w'(s)^2dt - 
	\int^t_{t_0}\left(\frac{N(N-2)}{8s^2} + \frac{N-2}{4}e^{-2s}\right)w(s)^2ds.
$$
Letting $t \to \infty$, we have 
$$
	\int^{\infty}_{t_0}\left(\frac{1}{2}-b(t)\right)w'(t)^2dt 
	= \zeta - {\cal E}(w)(t_0)+
	\int^{\infty}_{t_0}\left(\frac{N(N-2)}{8t^2} + \frac{N-2}{4}e^{-2t}\right)w(s)^2dt,
$$
which implies that $\int^{\infty}_{t_0}tw'(t)^2 dt < \infty$.
Thus (ii) holds.
\end{proof}

\begin{lemma}
\label{lem5.7}
Let $w$ be defined as in Proposition $\ref{prp5.5}$. 
Assume that the limit of $w(t)$ as $t \to \infty$ does not exist. 
Then the following {\rm (i)} and {\rm (ii)} hold.
\begin{itemize}
\item[{\rm (i)}] One has $\zeta > \Psi(K)$, where $\zeta$ 
is the constant in Lemma \ref{lem5.6} {\rm (i)}.

\item[{\rm (ii)}] There exists an infinite sequence $\tau_n \to \infty$ such that 
$w(\tau_n) = K$ for $n = 1, 2, \ldots$.
\end{itemize}
\end{lemma}

\begin{proof}
Since $w(t)$ is positive and bounded for $t \geq t_0$ by Lemma \ref{lem4.4}, 
there exist $0 \leq \gamma_1 < \gamma_2$ such that  
$$
	\liminf_{t \to \infty}w(t) = \gamma_1 
	\quad \mbox{and} \quad 
	\limsup_{t \to \infty}w(t) = \gamma_2.
$$
Then there exist sequences $t_n \to \infty$ and $s_n \to \infty$ such that 
\begin{equation}
	w'(t_n) = w'(s_n) = 0, \quad \lim_{n \to \infty}w(t_n) = \gamma_1, 
	\quad \mbox{and} \quad 
	\lim_{n \to \infty}w(s_n) = \gamma_2.
	\label{eq5.17}
\end{equation}
Define ${\cal E}(w)(t)$ by (\ref{eq2.10}). 
Then we obtain 
$$
	\lim_{n \to \infty}{\cal E}(w)(t_n) = \lim_{n \to \infty}\Psi(w(t_n)) 
	= \Psi(\gamma_1) 
	\quad \mbox{and} \quad  
	\lim_{n \to \infty}{\cal E}(w)(s_n) = \lim_{n \to \infty}\Psi(w(s_n)) 
	= \Psi(\gamma_2). 
$$
By Lemma \ref{lem5.6} (i), we have 
$\Phi(\gamma_1) = \Phi(\gamma_2) = \zeta = \lim_{t \to \infty}{\cal E}(w)(t)$.
Since $\Psi$ is defined by (\ref{eq2.12}) and $0 \leq \gamma_1 < \gamma_2$, we conclude that 
\begin{equation}
	\gamma_1 < K < \gamma_2 \quad \mbox{and} \quad 
	\Psi(K) < \zeta = \Psi(\gamma_1) = \Psi(\gamma_2).
	\label{eq5.18}
\end{equation}
From (\ref{eq5.17}) and (\ref{eq5.18}), there exists 
an infinite sequence $\tau_n \to \infty$ such that $w(\tau_n) = K$ for $n = 1, 2, \ldots$.
Thus (i) and (ii) hold.
\end{proof}

We are now in a position to prove Proposition \ref{prp5.5}.

\begin{proof}[Proof of Proposition \ref{prp5.5}] 
It suffices to show that $w(t) \to K$ as $t \to \infty$. 
First we show that 
\begin{equation}
	w(t) \to \gamma \quad \mbox{as} \ t \to \infty
	\label{eq5.19}
\end{equation}
with some $\gamma > 0$. 
Assume by contradiction that 
the limit of $w(t)$ as $t \to \infty$ does not exist.
Lemma \ref{lem5.7} implies that
$\lim_{t \to \infty}{\cal E}(w)(t) = \zeta > \Psi(K)$ and 
there exists a sequence $\tau_n \to \infty$ such that 
$w(\tau_n) = K$ for $n = 1, 2, \ldots$.
Then, from (\ref{eq2.15}) we obtain  
$$
	\lim_{n \to \infty}{\cal E}(w)(\tau_n) = 
	\lim_{n \to \infty}\left(\frac{\tau_nw'(\tau_n)^2}{2} + \Psi(K) + 
	\left(\frac{N(N-2)}{8\tau_n} + \frac{N-2}{8}e^{-2\tau_n}\right)K^2 \right) = \zeta.
$$
Then it follows that 
$$
	\lim_{n \to \infty}\frac{\tau_nw'(\tau_n)^2}{2} = \zeta - \Psi(K) > 0.
$$
Hence, there exists an integer $n_0$ such that 
$$
	\tau_n w'(\tau_n)^2 \geq \zeta-\Psi(K) \quad \mbox{for} \ n \geq n_0.
$$
Note that  
$$
	\frac{d}{dt}(tw'(t)^2) = w'(t)^2 + 2tw'(t)w''(t).
$$
Lemma \ref{lem4.4} implies that $(d/dt)(tw'(t)^2)$ is bounded for $t \geq t_0$. 
Then there exists $\rho > 0$ such that 
$$
	t w'(t)^2 \geq \frac{\zeta-\Phi(K)}{2} 
	\quad \mbox{for} \ \tau_n-\rho \leq t \leq \tau_n+ \rho
	\quad \mbox{with} \ n \geq n_0.
$$
This implies that 
$$
	\int^{\infty}_{t_0}tw'(t)^2dt = \infty,
$$
which contradicts Lemma \ref{lem5.6} (ii). 
Thus (\ref{eq5.19}) holds with some $\gamma \geq 0$. 
By Lemma \ref{lem4.5} we have $\gamma > 0$. 

We next we show that 
\begin{equation}
	\lim_{t \to \infty}tw'(t)^2 = 0.
	\label{eq5.20}
\end{equation}
Letting$t \to \infty$in (\ref{eq2.15}), 
from (\ref{eq5.19}) and Lemma \ref{lem5.6} (i) we have 
\begin{equation}
	\begin{array}{rcl}
	\dsp
	\lim_{t \to \infty}\frac{tw'(t)^2}{2} 
	& = & \dsp
	\lim_{t \to \infty}\left({\cal E}(w)(t) - \Psi(w(t)) -   
	\left(\frac{N(N-2)}{8t} + \frac{N-2}{8}e^{-2t}\right)w(t)^2
	\right) 
	\\[2ex]
	& = & \zeta - \Psi(\gamma) \geq 0.
	\end{array}
	\label{eq5.21}
\end{equation}
Then it suffices to show that $\zeta = \Phi(\gamma)$. 
Assume that $\zeta > \Phi(\gamma)$. 
Then, from (\ref{eq5.21}), we obtain   
$$
	|w'(t)| > \frac{\sqrt{2(\zeta - \Phi(\gamma))}}{2\sqrt{t}} > 0 \quad \mbox{for} \ t \geq t_1
$$
with some $t_1 \geq t_0$,  
which implies that $\lim_{t \to \infty}|w(t)| = \infty$. 
This is a contradiction. Thus we obtain $\zeta = \Phi(\gamma)$, 
and hence (\ref{eq5.20}) holds.

Finally, we will show that $\gamma = K$ in (\ref{eq5.19}). 
Assume by contradiction that $\gamma \neq K$. 
Letting $t \to \infty$ in (\ref{eq2.7}), 
from (\ref{eq5.19}) and (\ref{eq5.20}) we obtain  
$$
	\lim_{t \to \infty}(tw''(t) + (N-2)tw'(t)) = -\frac{(N-2)^2}{2}\gamma + \gamma^{p} \neq 0.
$$
Then we obtain 
$$
	|w''(t) + (N-2)w'(t)| \geq \frac{1}{2}\frac{|-\frac{(N-2)^2}{2}\gamma + \gamma^{p}|}{t} 
	\quad \mbox{for} \ t \geq t_2 
$$
with some $t_2 \geq t_0$. 
Integrating on $[t, t_2]$, and letting $t \to \infty$, 
we obtain $|w'(t) + (N-2)w(t)| \to \infty$ as $t \to \infty$. 
This is a contradiction. 
Thus we obtain $\gamma = K$ in (\ref{eq5.19}), and hence (\ref{eq5.16}) holds.
\end{proof}

\begin{proof}[Proof of Theorems \ref{thm1.1} and \ref{thm1.2}] 
In the cases $N/(N-2) < p < (N+2)/(N-2)$ and $p = N/(N-2)$, there exists 
infinitely many singular positive solutions of (\ref{eq1.2}) 
by Propositions \ref{prp3.1} and \ref{prp3.6}, respectively. 
By Propositions \ref{prp5.1} (i) and \ref{prp5.5}, 
any singular solution satisfies (\ref{eq1.4}) and (\ref{eq1.5}), 
respectively.
Thus Theorem \ref{thm1.1} holds.

In the case $p = (N+2)/(N-2)$, 
there exists infinitely many singular positive solutions of (\ref{eq1.2}) 
by Proposition \ref{prp3.1}.
By Proposition \ref{prp5.1} (ii), for each singular solution $u$, 
there exist constants $\gamma_1 \leq \gamma_2$ 
satisfying (\ref{eq1.11}) and (\ref{eq5.2}), and $u \equiv U_L$ if $\gamma_1 = \gamma_2$.  
Thus Theorem \ref{thm1.2} holds.
\end{proof}


\section{Bounds of the solutions at infinity: Proof of Theorem \ref{thm1.3}}

Throughout this subsection we assume that $p > (N+2)/N$. 
For simplicity, we define ${\cal L}u$ by 
$$
	{\cal L}u = u'' + \left(\frac{N-1}{r} + \frac{r}{2}\right)u' + \frac{1}{p-1}u
$$ 
for $u \in C^2(0, \infty)$. 
Then (\ref{eq1.2}) can be written as 
\begin{equation}
	{\cal L}u + u^p = 0 \quad \mbox{for} \ r > 0.
	\label{eq6.1}
\end{equation}
In order to prove Theorem \ref{thm1.3}, we need the following proposition. 

\begin{proposition}
\label{prp6.1}
Assume that there exists a positive function $v \in C^2(0, \infty)$ 
satisfying 
\begin{equation}
	{\cal L}v + v^p \leq 0 \quad \mbox{for} \ r > 0, \quad 
	\lim_{r \to 0}v(r)  = \infty \quad 
	\mbox{and} \quad 
	\lim_{r \to \infty}r^{2/(p-1)}v(r)  = \ell
	\label{eq6.2}
\end{equation}
with some $\ell > 0$. 
Then $(\ref{eq6.1})$ has a positive solution $\underline{u} \in C^2(0, \infty)$ satisfying 
\begin{equation}
	0 < \underline{u}(r) \leq v(r) \quad \mbox{for} \ r > 0 
	\quad \mbox{and} \quad \lim_{r \to \infty}r^{2/(p-1)}\underline{u}(r) = \ell.
	\label{eq6.3}
\end{equation}
Furthermore, if $\underline{u}$ is bounded near the origin, then 
$\underline{u} \in C^1[0, \infty)$ and satisfies $\underline{u}'(0) = 0$.
\end{proposition}

In order to prove Proposition \ref{prp6.1}, we need some lemmas. 

\begin{lemma}
\label{lem6.2}
\begin{itemize}
\item[{\rm (i)}] Assume that $u$ satisfies 
$$
	{\cal L}u \leq 0 \quad \mbox{for} \ r > 0 
	\quad \mbox{and} \quad \liminf_{r \to \infty}r^{2/(p-1)}u(r) \geq 0.
$$
Assume in addition that either $u'(0) = 0$ or $\lim_{r \to 0}u(r) = \infty$. 
Then $u(r) \geq 0$ for $r \geq 0$.

\item[{\rm (ii)}] 
Assume that $u$ satisfies 
${\cal L}u \leq 0$ for $0 < r \leq R$ and $u(R) \geq 0$.
Assume in addition either that $u'(0) = 0$ or $\lim_{r \to 0}u(r) = \infty$. 
Then $u(r) \geq 0$ for $0 \leq r \leq R$.
\end{itemize}
\end{lemma}

\begin{proof}
Define $\phi_0(r) = e^{-r^2/4}$ for $r \geq 0$. Then $\phi_0$ satisfies 
$$
	\phi_0'' + \left(\frac{N-1}{r} + \frac{r}{2}\right)\phi_0' + \frac{N}{2}\phi_0 = 0
	\quad \mbox{for} \ r > 0.
$$
We consider the case where $u'(0) = 0$. 
Define $w(x) = u(|x|)/\phi(|x|)$ with $x \in \R^N$. 
Then $w$ satisfies
\begin{equation}
	\Delta w + \left(2\frac{\nabla \phi_0}{\phi_0}+ \frac{1}{2}\phi_0 x\right)\cdot 
	\nabla w - \left(\frac{N}{2}-\frac{1}{p-1}\right)w \leq 0 
	\quad \mbox{in} \ \R^N.
	\label{eq6.4}
\end{equation}
It suffices to show that $w(x) \geq 0$ for all $x \in \R^N$. 
Assume by contradiction that there exists $x_1 \in \R^N$ such that $w(x_1) < 0$. 
Then there exists $x_2 \in \R^N$ such that $w(x_2) = \min_{x \in \R^N}w(x)$, 
and hence we have $\Delta w(x_2) \geq 0$, $\nabla w(x_2) = 0$ and $w(x_2) < 0$, 
which contradicts (\ref{eq6.4}) since $N/2-1/(p-1) < 0$. 
Thus we obtain $w(x) \geq 0$ in $\R^N$, and hence $u(r) \geq 0$ for $r \geq 0$. 
In the case where $\lim_{r \to 0}u(r) = \infty$, 
the conclusion follows by a slight modification of the above argument. Thus (i) holds. 
Since the proof of (ii) is similar to (i), we omit it.
\end{proof}

Following the argument in \cite[Lemma 2.1]{Naia}, 
we will show the existence of positive solution $\phi$ of ${\cal L}\phi = 0$ for $r > 0$.

\begin{lemma}
\label{lem6.3}
For $\ell > 0$, there exists a positive solution 
$\phi_{\ell} \in C^2(0, \infty)\cap C^1[0, \infty)$ of the problem
\begin{equation}
	{\cal L}\phi_{\ell} = 0 \quad \mbox{for} \ r > 0, \quad  
	\phi_{\ell}'(0) = 0 \quad \mbox{and} \quad \lim_{r \to \infty}r^{2/(p-1)}\phi_{\ell}(r) = \ell.
	\label{eq6.5}
\end{equation}
\end{lemma}

\begin{proof}
Define 
\begin{equation}
	w(x, t) = \frac{1}{(4\pi t)^{N/2}}\int_{\R^N}e^{-|x-y|^2/(4t)}|y|^{-2/(p-1)}dy
	\quad \mbox{for} \ x \in \R^N, \ t > 0.
	\label{eq6.6}
\end{equation}
Note that $|\cdot|^{-2/(p-1)} \in L^1_{\rm loc}(\R^N)$ since $p > (N+2)/N$. 
Then $w$, defined by (\ref{eq6.6}), satisfies $w_t = \Delta w$ in $\R^N\times (0, \infty)$ and 
\begin{equation}
	w(x, t) \to |x|^{-2/(p-1)} \quad 
	\mbox{in}  \ L^1_{\rm loc}(\R^N) \quad \mbox{as} \ t \to 0.
	\label{eq6.7}
\end{equation}
From (\ref{eq6.6}) we see that $w(x, t) = \mu^{2/(p-1)}w(\mu x, \mu^2 t)$ for all $\mu > 0$. 
Letting $\mu = 1/\sqrt{t}$, we obtain 
\begin{equation}
	w(x,  t) = t^{-1/(p-1)}\phi(x/\sqrt{t}), 
	\label{eq6.8}
\end{equation}
where $\phi(y) = w(y, 1)$. 
It can be easily choked that $\phi \in C^2(\R^N)$ is radially symmetric and satisfies
$$
	\Delta \phi + \frac{1}{2}x\cdot \nabla \phi + \frac{1}{p-1}\phi = 0 
	\quad \mbox{in} \ \R^N.
$$
We denote $\phi = \phi(r)$ with $r = |y|$. 
Then $\phi(r)$ solves ${\cal L}\phi = 0$ for $r > 0$ and $\phi'(0) = 0$. 
From (\ref{eq6.7}) and (\ref{eq6.8}), for $x \neq 0$ we have  
$$
	(|x|/\sqrt{t})^{2/(p-1)}\phi(|x|/\sqrt{t}) = |x|^{2/(p-1)}w(x, t) \to 1 
	\quad \mbox{as} \ t \to 0. 
$$
This implies that $\lim_{r \to \infty}r^{2/(p-1)}\phi(r) = 1$. 
Then $\phi_{\ell}(r) = \ell\phi(r)$ solves (\ref{eq6.5}).
\end{proof}

\begin{lemma}
\label{lem6.4}
Let $g \in C[0, \infty)$ satisfy $g(r) \geq 0$ for all $r \geq 0$. 
Assume that there exists a positive function $v \in C^2(0, \infty)$ 
satisfying 
$$
	{\cal L}v + g(r) \leq 0 \quad \mbox{for} \ r > 0, 
	\quad \lim_{r \to 0}v(r) = \infty 
	\quad \mbox{and} \quad 
	\lim_{r \to \infty}r^{2/(p-1)}v(r) = 0.
$$ 
Then there exists a unique solution $u \in C^2(0, \infty)\cap C^1[0, \infty)$ 
of the problem
\begin{equation}
	\left\{
	\begin{array}{l}
	{\cal L}u + g(r) = 0 \quad \mbox{for} \ r > 0,
	\\[1ex]
	u'(0) = 0 \quad \mbox{and} \quad 0 \leq u(r) \leq v(r) 
	\quad \mbox{for} \ r > 0.
	\end{array}
	\right.
	\label{eq6.9}
\end{equation}
\end{lemma}

\begin{proof}
First we note that, for any $R > 0$, the problem
\begin{equation}
	\left\{
	\begin{array}{c}
	{\cal L}u + g(r) = 0 \quad \mbox{for} \ 0 < r < R,
	\\[1ex]
	u'(0) = 0 \quad \mbox{and} \quad u(R) = v(R).
	\end{array}
	\right.
	\label{eq6.10}
\end{equation}
has a solution $u \in C^2(0, R]\cap C^1[0, R]$ satisfying 
\begin{equation}
	0 \leq u(r) \leq v(r) \quad \mbox{for} \ 0 < r \leq R.
	\label{eq6.11}
\end{equation}
In fact, let $w$ be a solution of ${\cal L}w + g(r) = 0$ for $r > 0$ 
satisfying $w'(0) = 0$ and $w(0) = 0$. 
Take a constant $C \in \R$ such that $w(R) + C\phi_{\ell}(R) = v(R)$, where 
$\phi_{\ell}$ is the solution of (\ref{eq6.5}) obtained by Lemma \ref{lem6.3}. 
Then $u(r) = w(r) + C\phi_{\ell}(r)$ solves (\ref{eq6.10}).
Since $u$ satisfies ${\cal L}u = -g(r) \leq 0$, $u'(0) = 0$ and $u(R) = v(R) > 0$, 
from Lemma \ref{lem6.2} we have $u(r) \geq 0$ for $0 \leq r \leq R$. 
Let $w(r) = v(r) -u(r)$. 
Then $w$ satisfies ${\cal L}w \leq 0$,  $w(r) \to \infty$ as $r \to 0$ and  $w(R) = 0$. 
From Lemma \ref{lem6.2}, we have $w(r) \geq 0$ for $0 < r \leq R$, 
which implies that $u(r) \leq v(r)$ for $0 < r \leq R$.  
Thus (\ref{eq6.11}) holds.

Let $\{r_n\}$ be a sequence such that $r_n \to \infty$ as $n \to \infty$. 
By the above argument, for each $n = 1, 2, \ldots$, the problem
\begin{equation}
	\left\{
	\begin{array}{c}
	{\cal L}u + g(r) = 0 \quad \mbox{for} \ 0 < r < r_n,
	\\[1ex]
	u'(0) = 0 \quad \mbox{and} \quad u(r_n) = v(r_n),
	\end{array}
	\right.
	\label{eq6.12}
\end{equation}
has a solution $u_n \in C^2(0, r_n]\cap C^1[0, r_n]$ satisfying 
$0 \leq u_n(r) \leq v(r)$ for $0 \leq r \leq r_n$.  
Furthermore, for each $n = 1, 2, \ldots$, we have 
\begin{equation}
	u_{n+1}(r) \leq u_n(r) \quad 
	\quad \mbox{for} \ 0 \leq r \leq r_n. 
	\label{eq6.13}
\end{equation}
In fact, define $z_n(r) = u_{n}(r) -u_{n+1}(r)$ for $0 \leq r \leq r_n$. 
Then $z_n$ satisfies $Lz_n = 0$, $z_n'(0) = 0$ and $z_n(r_n) = v(r_n) -u_{n+1}(r_n) \geq 0$. 
By Lemma \ref{lem6.2}, we have $z_n(r) \geq 0$ for $0 \leq r \leq r_n$.
Thus we obtain (\ref{eq6.13}).  

Take any $R > 0$. 
For sufficiently large $n$ satisfying $r_n > R$, the solution $u_n$ of (\ref{eq6.12}) satisfies 
$$
	u_n'(r) = -\frac{1}{\rho(r)}\int^r_0 \rho(s)\left(\frac{1}{p-1}u_n(s) + g(s)\right)ds < 0 
	\quad \mbox{for} \ 0 < r \leq R,
$$
where $\rho(r) = r^{N-1}e^{r^2/4}$. 
Then we see that $0 < u_n(r) \leq u_n(0) \leq u_1(0)$ for $0 \leq r \leq R$ and 
$$
	|u_n'(r)| \leq \int^r_0\left(\frac{1}{p-1}u_n(s)+  g(s)\right)ds 
	\leq \int^R_0\left(\frac{1}{p-1}u_1(0)+  g(s)\right)ds 
	\quad \mbox{for} \  0 \leq r \leq R.
$$
Then $\{u_n\}$ is uniformly bounded and equicontinuous on $[0, R]$. 
By the Arzel\'a-Ascoli Theorem, there exist $u_R \in C[0, R]$ and a subsequence, 
which we denote by $\{u_n\}$ again, such that $u_n \to u_R$ uniformly on $[0, R]$. 
Note that $u_n$ satisfies 
$$
	u_n(r) = u_n(0) - \int^r_0 \frac{1}{\rho(s)}\left(\int^s_0 \rho(t)
	\left(\frac{1}{p-1}u_n(s) + g(s)\right)dt\right)ds 
	\quad \mbox{for} \  0 <  r \leq R.
$$
Letting $n \to \infty$, we have 
$$
	u_R(r) = u_R(0) - \int^r_0 \frac{1}{\rho(s)}\left(\int^s_0 \rho(t)
	\left(\frac{1}{p-1}u_R(t) + g(t)\right)dt\right)ds
	\quad \mbox{for} \  0 <  r \leq R.
$$
This implies that $u_R \in C^2(0, R]\cap C^1[0, R]$ satisfies 
${\cal L}u_R + g(r) = 0$ for $0 < r < R$. 
Furthermore, we obtain $0 \leq u_R(r) \leq v(r)$ for $0 < r \leq R$ and 
$u_R(0) \leq u_1(0)$. 

Let $\{R_n\}$ be a sequence satisfying $R_n \to \infty$ as $n \to \infty$. 
By the diagonal argument, we conclude that there exists a solution 
$u \in C^2(0, \infty)\cap C^1[0, \infty)$ of (6.9). 

We will show that (\ref{eq6.10}) has a unique solution. 
Assume that (\ref{eq6.9}) has solutions $u_1$ and $u_2$. 
Since $\lim_{r \to \infty}r^{2/(p-1)}v(r) = 0$, we have 
\begin{equation}
	\lim_{r \to \infty}r^{2/(p-1)}u_i(r) = 0 \quad \mbox{for} \ i = 1, 2.
	\label{eq6.14}
\end{equation}
Let $\phi(r) = u_1(r)-u_2(r)$. Then $\phi$ satisfies 
\begin{equation}
	{\cal L}\phi = 0 \quad \mbox{for} \ r > 0
	\quad \mbox{and} \quad \phi'(0) = 0.
	\label{eq6.15}
\end{equation}
By the uniqueness for the ODE problem,  
any solution of (\ref{eq6.15}) must be a constant multiple of 
$\phi_{\ell}$, where $\phi_{\ell}$ is the solution of the problem (\ref{eq6.5}). 
Then we have $\phi(r) = C\phi_{\ell}(r)$ with some constant $C \in \R$. 
From (\ref{eq6.14}) we have $\lim_{r \to \infty}r^{2/(p-1)}\phi(r) = 0$, 
which implies that $C = 0$, and hence $\phi(r) \equiv 0$. 
Thus the solution $u$ of (\ref{eq6.9}) is unique.
\end{proof}

Let us consider the equation 
\begin{equation}
	{\cal L}\hat{u} + (\hat{u} + \phi_{\ell})^p = 0  
	\quad \mbox{for} \ r > 0,
	\label{eq6.16} 
\end{equation}
where $\phi_{\ell}$ is the solution obtained by Lemma \ref{lem6.3}. 
Assume that there exists a positive function 
$\hat{v} \in C^2(0, \infty)$ satisfying 
\begin{equation}
	{\cal L}\hat{v} + (\hat{v} + \phi_{\ell})^p \leq 0 
	\quad \mbox{for} \  r > 0, \quad 
	\lim_{r \to \infty}\hat{v}(r) = \infty 
	\quad \mbox{and} \quad 
	\lim_{r \to \infty}r^{2/(p-1)}\hat{v}(r) = 0.
	\label{eq6.17}
\end{equation}
For each $u \in C[0, \infty)$ satisfying $0 \leq u(r) \leq \hat{v}(r)$ for $r > 0$, 
define the mapping $v = Tu$ by setting $v \in C^2(0, \infty)\cap C^1[0, \infty)$ 
to be a unique solution of 
\begin{equation}
	\left\{
	\begin{array}{l}
	{\cal L}v + (u + \phi_{\ell})^p = 0 \quad \mbox{for} \ r > 0, 
	\\[1ex]
	\dsp
	v'(0) = 0 \quad  \mbox{and} \quad 0 \leq v(r) \leq \hat{v}(r) \quad 
	\mbox{for} \ r > 0.
	\end{array}
	\right.
	\label{eq6.18}
\end{equation}
Assume that $u \in C[0, \infty)$ satisfies $0 \leq u(r) \leq \hat{v}(r)$ for $r > 0$.
Let $g(r) = (u(r) + \phi_{\ell}(r))^p$. 
Since $g(r) \leq (\hat{v}(r) + \phi_{\ell}(r))^p$, 
we have ${\cal L}\hat{v} + g(r) \leq 0$ for $r > 0$. 
By Lemma \ref{lem6.4}, there exists a unique solution $v$ of (\ref{eq6.18}).
Thus the mapping $T$ is well defined.

\begin{lemma}
\label{lem6.5}
Assume that $u_1, u_2 \in C[0, \infty)$ satisfy $0 \leq u_1(r) \leq u_2(r) \leq \hat{v}(r)$ for $r \geq 0$. 
Let $v_i = Tu_i$ for $i = 1, 2$. 
Then $v_1(r) \leq v_2(r)$ for $r \geq 0$.
\end{lemma}

\begin{proof}
Define $w(r) = v_2(r)-v_1(r)$ for $r \geq 0$. 
Then $w$ satisfies ${\cal L}w \leq 0$ for $r \geq 0$ and $w'(0) = 0$ and 
$\lim_{r\to \infty}r^{2/(p-1)}w(r) = 0$. 
Then, from Lemma \ref{lem6.2} we have $w(r) \geq 0$ for $r \geq 0$, 
and hence we have $v_2(r) \geq v_1(r)$ for $r \geq 0$.
\end{proof}

\begin{lemma}
\label{lem6.6}
Assume that there exists a positive function $\hat{v} \in C^2(0, \infty)$ 
satisfying $(\ref{eq6.17})$. 
Then the equation $(\ref{eq6.16})$ has a solution $\hat{u} \in C^2(0, \infty)$ 
satisfying
\begin{equation}
	0 < \hat{u}(r) \leq \hat{v}(r) \quad \mbox{for} \ r \geq 0.
	\label{eq6.19}
\end{equation}
\end{lemma}

\begin{proof}
Set $\hat{u}_0 \equiv 0$, and define $\hat{u}_k = T\hat{u}_{k-1}$ for $k = 1, 2, \ldots$. 
By induction, we find that $\hat{u}_k$ is well defined and satisfies 
$$
	0 \equiv \hat{u}_0 < \hat{u}_1 < \hat{u}_2 < \cdots \hat{u}_k < \hat{u}_{k+1} < 
	\cdots < \hat{v}.
$$
Define $\hat{u}(r) = \lim_{k \to \infty}\hat{u}_k(r)$ for $r > 0$. 
Then $0 < \hat{u}(r) \leq \hat{v}(r)$ for $r > 0$. 
Take any $r_0 > 0$. 
For each $k = 1, 2, \ldots$, we have  
$$
	\hat{u}_k(r) = \hat{u}_k(r_0) -\int^r_{r_0} \frac{1}{\rho(s)}\int^s_{0} \rho(t)
	\left(\frac{1}{p-1}\hat{u}_k(t) + (\hat{u}_{k-1}(t) + \phi_{\ell}(t))^p\right) dtds
	\quad \mbox{for} \ r \geq r_0.
$$
Letting $k \to \infty$, by the monotone convergence theorem, we obtain 
$$
	\hat{u}(r) = \hat{u}(r_0) -\int^r_{r_0} \frac{1}{\rho(s)}\int^s_{0} \rho(t)
	\left(\frac{1}{p-1}\hat{u}(t) + (\hat{u}(t) + \phi_{\ell}(t))\right)dtds
	\quad \mbox{for} \ r \geq r_0.
$$
Thus $\hat{u}$ satisfies (\ref{eq6.16}) for $r > r_0$. 
Since $r_0 > 0$ is arbitrary, we obtain a solution $\hat{u} \in C^2(0, \infty)$ of (\ref{eq6.16}).
\end{proof}

\begin{proof}[Proof of Proposition \ref{prp6.1}] 
Define $\hat{v}(r) = v(r) - \phi_{\ell}(r)$. 
From (\ref{eq6.2}) we have ${\cal L}\hat{v} \leq -v^p < 0$ for $r > 0$, 
$\lim_{r \to 0}\hat{v}(r) = \infty$ and $\lim_{r \to \infty}r^{2/(p-1)}\hat{v}(r) = 0$. 
Then, from Lemma \ref{lem6.2}, we have $\hat{v}(r) \geq 0$ for $r > 0$. 
Assume that there exists $r_1 > 0$ such that $\hat{v}(r_1) = 0$. 
Then $\hat{v}'(r_1) = 0$ and $\hat{v}''(r_1) \geq 0$, and hence 
${\cal L}\hat{v}(r_1) \geq 0$, which is a contradiction. 
Thus $\hat{v}(r)$ is positive for $r > 0$ and satisfies (\ref{eq6.17}). 
Applying Lemma \ref{lem6.5}, we have a solution $\hat{u}$ of (\ref{eq6.16}) satisfying (\ref{eq6.19}). 
Define  $\underline{u}(r) = \hat{u} + \phi_{\ell}$. 
Then $\underline{u}$ is positive solution of (\ref{eq6.1}) satisfying (\ref{eq6.3}).
Lemma \ref{lem4.1} implies that $\underline{u}'(r) \leq 0$ for $r > 0$.
In the case $\underline{u}(r)$ be bounded near the origin, 
we obtain $\underline{u}'(0) = 0$ by the usual argument. So we omit its proof.
\end{proof}

\begin{proof}[Proof of Theorem \ref{thm1.3}] 
Let $u$ be a singular positive solution of (\ref{eq1.2}), 
and let $\ell \geq 0$ be the limit in (\ref{eq1.15}). 
Take any $\mu \in (0, 1)$, and define $v(r) = \mu u(r)$. 
Then $v$ satisfies 
$$
	{\cal L}v + v^p < 0 \quad \mbox{for} \ r > 0, \quad 
	\lim_{r \to 0}v(r) =\infty \quad \mbox{and} \quad 
	\lim_{r \to \infty}r^{2/(p-1)}v(r) = \mu\ell.
$$
Applying Proposition \ref{prp6.1}, we have a positive solution 
$\underline{u} \in C^2(0, \infty)$ of (\ref{eq6.1}) satisfying 
\begin{equation}
	0 < \underline{u}(r) \leq v(r) \quad \mbox{for} \ r > 0 
	\quad \mbox{and} \quad \lim_{r \to \infty}r^{2/(p-1)}\underline{u}(r) = \mu\ell.
	\label{eq6.20}
\end{equation}
Lemma \ref{lem4.1} implies that $\underline{u}'(r) \leq 0$ for $r > 0$.

We will show that $\underline{u}(r)$ is bounded near the origin.
Assume by contradiction that $\underline{u}(r) \to \infty$ as $r \to 0$. 
We first consider the case $N/(N-2) < p < (N+2)/(N-2)$. 
Theorem \ref{thm1.1} implies that the solutions $u(r)$ and $\underline{u}(r)$ satisfy 
$$
	\lim_{r \to 0}r^{2/(p-1)}u(r) = L 
	\quad \mbox{and} \quad 
	\lim_{r \to 0}r^{2/(p-1)}\underline{u}(r) = L, 
$$
where $L$ is the constant defined in (\ref{eq1.3}). 
On the other hand, from (\ref{eq6.20}), we have
$$
	 \lim_{r \to 0}r^{2/(p-1)}\underline{u}(r) \leq 
	 \lim_{r \to 0}r^{2/(p-1)}v(r) = \mu L < L,
$$
which is a contradiction. 
In the case $p = N/(N-2)$, we can reach a contradiction by the same argument as above. 
So we omit the proof in this case. 
We finally consider the case $p = (N+2)/(N-2)$.
From Theorem \ref{thm1.2}, the solution $u(r)$ satisfies 
$$
	\gamma_1 = \liminf_{r \to 0}r^{2/(p-1)}u(r) \leq
	\limsup_{r \to 0}r^{2/(p-1)}u(r) = \gamma_2
$$
with some constants $\gamma_1$ and $\gamma_2$ satisfying 
$\gamma_1 \leq L \leq \gamma_2$ and $\Phi(\gamma_1) = \Phi(\gamma_2)$, 
where $\Phi$ is defined by (\ref{eq1.8}). 
From Theorem \ref{thm1.2}, the solution $\underline{u}(r)$ also satisfies 
$$
	\underline{\gamma}_1 = 
	\liminf_{r \to 0}r^{2/(p-1)}\underline{u}(r) \leq
	\limsup_{r \to 0}r^{2/(p-1)}\underline{u}(r) 
	= \underline{\gamma}_2
$$
with some constants $\underline{\gamma}_1$ and $\underline{\gamma}_2$ satisfying 
$\underline{\gamma}_1 \leq L \leq \underline{\gamma}_2$ 
and $\Phi(\underline{\gamma}_1) = \Phi(\underline{\gamma}_2)$. 
On the other hand, from (6.20) we have 
$$
	\underline{\gamma}_1 \leq \mu \gamma_1 < L \quad 
	\mbox{and} \quad 
	L \leq \underline{\gamma}_2 \leq \mu \gamma_2.
$$
Since $\Phi(v)$ is decreasing for $0 < v < L$ and increasing for $v > L$,
we have
$$
	\Phi(\underline{\gamma}_1) \geq \Phi(\mu\gamma_1) > \Phi(\gamma_1)
	\quad \mbox{and} \quad 
	\Phi(\underline{\gamma}_2) \leq \Phi(\mu\gamma_2) < \Phi(\gamma_2).
$$
From $\Phi(\gamma_1) = \Phi(\gamma_2)$ we obtain 
$\Phi(\underline{\gamma}_1) > \Phi(\underline{\gamma}_2)$, 
which is a contradiction. 
Thus $\underline{u}(r)$ is bounded near $0$ in any cases $N/(N-2) \leq p \leq (N+2)/(N-2)$.
Then $\underline{u} \in C^2(0, \infty)\cap C^1[0, \infty)$ and satisfies $\underline{u}'(0) = 0$.
From (\ref{eq6.20}) we have $S_{\mu\ell} \neq \emptyset$, which implies that  $\mu\ell \leq \ell^*$. 
Since $\mu \in (0, 1)$ is arbitrary, we obtain $\ell \leq \ell^*$.
\end{proof}


\section*{Acknowledgements}
The author was supported by JSPS KAKENHI Grant Number JP26K06881.
This work was also supported by Research Institute for Mathematical Sciences, a Joint
Usage/Research Center located in Kyoto University.

\section*{Data availability} 
No data was used for the research described in the article.


\end{document}